\RequirePackage{fix-cm}
\documentclass[smallcondensed]{svjour3}     
\smartqed  
\usepackage{subcaption}
\usepackage{graphicx}
\usepackage{multirow}%
\usepackage{amsmath,amssymb,amsfonts}%
\usepackage{mathrsfs}%
\usepackage[title]{appendix}%
\usepackage{xcolor}%
\usepackage{placeins}
\usepackage{textcomp}
\usepackage{manyfoot}%
\usepackage{booktabs}%
\usepackage{comment}
\usepackage{mathtools}
\newcommand{\A}{\mathcal{A}}

\newcommand{\x}{\mathbf{x}}
\newcommand{\g}{\mathbf{g}}

\newcommand{\vve}{\mathbf{v}}
\newcommand{\Cind}{C_{\indep}}
\newcommand{\EF}{F_{\E}}
\newcommand{\EG}{G_{\E}}
\newcommand{\f}{\mathbf{f}}

\newcommand{\Sym}{\mathcal{S}}

\newcommand{\R}{\mathbb{R}}

\newcommand{\s}{\mathbf{s}}

\newcommand{\M}{\mathcal{M}}
\newcommand{\F}{\mathcal{F}}
\newcommand{\E}{\mathbb{E}}
\newcommand{\N}{\mathbb{N}}
\newcommand{\rb}{]}
\newcommand{\lb}{[}
\newcommand{\D}{\mathcal{D}}
\newcommand{\Tr}{\text{Tr }}
\newcommand{\ve}{\varepsilon}
\newcommand{\e}{\mathbf{e}}

\newcommand{\Ci}{C_{\text{init}}}
\newcommand{\G}{\mathcal{G}}
\newcommand{\NN}{[ N]}

\newcommand{\perf}{\mathcal{P}}
\newcommand{\indep}{\perp\kern-5pt\perp}
\renewcommand{\th}{\text{th}}

\newcommand{\Ge}{\mathcal{G}_E}
\newcommand{\Gc}{\mathcal{G}_c}
\newcommand{\Cn}{C_{\text{variance}}}

\begin{document}
\title{Computer-aided analyses of stochastic first-order methods, via interpolation conditions for stochastic optimization\thanks{A. Rubbens and S. Colla are, respectively, Research Fellow and FRIA grantees of the Fonds de la Recherche Scientifique - FNRS.  J.\!~Hendrickx is supported by the SIDDARTA Concerted Research Action (ARC) of the Federation Wallonie-Bruxelles.}}
\titlerunning{Interpolation conditions for stochastic optimization.}
\author{Anne Rubbens\and Sébastien Colla \and Julien M. Hendrickx\footnotemark[1]}


\institute{ICTEAM Institute, UCLouvain, Belgium. name.surname@uclouvain.be
}

\date{Received: date / Accepted: date}

\maketitle

\begin{abstract}
This work proposes a framework, embedded within the Performance Estimation framework (PEP), for obtaining worst-case performance guarantees on stochastic first-order methods. Given a first-order method, a function class, and a noise model with prescribed expectation and variance properties, we present a semidefinite program (SDP), whose size grows linearly with $N$, the number of iterations analyzed, and whose solution yields a convergence guarantee on the problem. 

The framework accommodates a wide range of stochastic settings, with finite or infinite support, including the unstructured noise model with bounded variance, finite-sum optimization, and block-coordinate methods, in a unified manner, as guarantees apply to any setting consistent with the noise model, i.e., its expectation and variance. It covers both non-variance-reduced and variance-reduced methods. Using the framework, we analyze the stochastic gradient method under several noise models, and illustrate how the resulting numerical and analytical convergence rates connect with existing results. In particular, we provide improved convergence rates on the unstructured noise model with bounded variance and in the block-coordinate setting.
\end{abstract}

\section{Introduction}
\label{sec:introduction}

Consider an unconstrained minimization problem
\begin{equation}\label{eq:minimization_problem}
    \min_{x\in \R^d} f(x),
\end{equation}
where $f:\R^d\to\R$ is a differentiable convex function. To solve (\ref{eq:minimization_problem}), consider a stochastic first-order black-box optimization method, i.e., a method that relies on evaluations of $f$ and \emph{noisy} evaluations of its gradient at a sequence of iterates. Such methods include, e.g., the stochastic gradient method (SGD), whose iteration is given by
\begin{align} \label{eq:SGD}\tag{SGD}
  x_{k+1}=x_k-\alpha_k(\nabla f(x_k)+\varepsilon_k),
\end{align}
where $\varepsilon_k$ is a zero-mean random variable, independent across queries.

Stochastic first-order methods, as opposed to exact first-order methods, i.e., $\varepsilon_k=0, \ \forall k$, naturally arise in a wide range of contexts where one has access to (or prefers to compute), e.g., an unbiased estimator of the gradient of $f$ rather than the gradient itself. Depending on these contexts, the distribution of $\varepsilon_k$ satisfies different assumptions. 
\paragraph{Additive noise on deterministic gradients}
The simplest setting arises when minimizing a deterministic function whose gradient evaluations are corrupted by additional zero-mean stochastic noise. This noise can be either \emph{relative}, with variance depending on the gradient norm 
or \emph{absolute}, with variance independent of the gradient norm. 

\paragraph{Stochastic optimization problems}
Another example is the minimization of an expectation, i.e.,
\begin{align}\label{eq:stoch_opti_problems}
    f(x) = \E_{\xi \sim \D}[f_\xi(x)],
\end{align}
where samples $\xi \sim \D$ are available, but the distribution $\D$ itself is unknown. In such cases, gradients of $f$ are unavailable, contrary to gradients of $f_\xi$, which provide an unbiased estimator of $\nabla f$. Hence, methods rely on unbiased noisy evaluations $\nabla f_\xi(x_k) = \nabla f(x_k) + \varepsilon_k$.

Stochastic optimization problems are particularly relevant in statistical supervised learning. Indeed, letting $\D$ be the unknown distribution linking inputs and outputs, $\xi$ be an input-output sample drawn from $\D$, $x$ a prediction model, and $f_\xi(x)$ the loss incurred by $x$ on $\xi$, minimizing $f$ corresponds to minimizing the expected risk or loss of $x$; see, e.g., \cite[Section~3.1]{bottou2018optimization}.
\paragraph{Finite-sum minimization}
One of the most studied categories of problems relying on stochastic optimization consists of the minimization of a large sum of functions, i.e.,
\begin{align}\label{eq:functions_sum}
    f(x) = \sum_{i=1}^n f_i(x),
\end{align}
where $n$ is large and $f_i$'s are differentiable convex functions. In contrast with the first two examples, exact computation of $\nabla f$ is possible in this case, but is computationally expensive. On the other hand, choosing uniformly at random an index $i \in \{1,2,\ldots,n\}$, furnishes $\nabla f_i(x)$, an unbiaised estimator of $\nabla f(x)$, whose computation is approximately $n$ times cheaper than the one of $\nabla f(x)$. As shown, e.g., in \cite{bottou2007tradeoffs}, it can therefore be advantageous to minimize \eqref{eq:functions_sum} using stochastic methods that rely on $\nabla f_i(x_k) = \nabla f(x_k) + \varepsilon_k$, instead of $\nabla f(x_k)$. Alternative ways to construct estimators of $\nabla f$ in this setting, leading to different assumptions on $\varepsilon_k$, are discussed in Section~\ref{sec:related_works}.

Such minimization problems typically arise as approximations of \eqref{eq:stoch_opti_problems} using $n$ independent samples of $\xi$, see, e.g., \cite[Section~3.1]{bottou2018optimization}.
\paragraph{Randomized block-coordinate descent methods}
Finally, another example consists of minimizing large-dimensional differentiable functions, that exhibit component-wise structure. In this setting, it is again possible to compute the exact gradient, but it is computationally advantageous to evaluate the gradient only along a single coordinate, chosen uniformly at random.  In this case, methods sample $d\e_i^\top\nabla f(x_k) \e_i=\nabla f(x_k)+\varepsilon_k$, where $\e_i$ is the $i$-th unit vector, $d$ is the dimension of $x$, and $\ve_k$ is zero-mean. Such methods are known as \emph{randomized coordinate descent} (RCD) methods, and can be generalized to gradient estimators along blocks of coordinates; see, e.g., \cite{nesterov2012efficiency,beck2013convergence}.
\\~\\
\noindent\textbf{ Goal}. Given a function class $\F$, a stochastic first-order black-box method $\M$, and the expectation and variance of the estimator of the gradient $\M$ has access to, we aim at obtaining a worst-case guarantee on the performance of $\M$ on $\F$, i.e., a bound on the accuracy of the solution after $N$ iterations of $\M$, valid for all functions in $\F$ and all distributions of $\varepsilon$ satisfying the given assumptions. In particular, we say the bound is \emph{tight} when attained by a specific function under a specific noise distribution.

\subsection{Related works}\label{sec:related_works}
\paragraph{Stochastic first-order methods}
First-order stochastic methods, first conceptualized in \cite{robbins1951stochastic}, have since then been the subject of extensive research, see, e.g., \cite{borkar2008stochastic,bottou2018optimization,garrigos2023handbook,gower2019sgd,netrapalli2019stochastic} and references therein. Analyses depend on (i) the method, (ii) the function class considered, and (iii) the noise model, that is, at minima, its variance and expectation, and, in some structural settings as the finite sum and the block-coordinate settings, additional information on its distribution. 

A widely analyzed model, though restrictive, is the additive bounded noise model ($\E_{\ve_k}[\|\ve_k\|^2]\leq \sigma^2$), see, e.g., \cite{nemirovski1994efficient,nemirovski2009robust,nemirovskij1983problem}. Under this model, methods ranging from SGD, with or without iterates averaging \cite{polyak1992acceleration,ruppert1988efficient}, to methods involving momentum \cite{xiao2009dual,ghadimi2012optimal,pmlr-v80-cohen18a,cyrus2018robust,devolder2013first} have been designed and analyzed on smooth convex functions.

Stochastic first-order methods have also been extensively analyzed in the finite-sum setting $f=\frac{1}{n}\sum_if_i$. Analyses depend on the properties of the $f_i$'s, and on the way an estimator of $\nabla f(x_k)$ is constructed. Classical assumptions include convexity and $L_i$-smoothness of each $f_i$, and boundedness of the variance at the optimum, i.e., $\E[\|f_i(x_\star)\|]=\sigma_\star^2\leq+\infty$, where $x_\star$ is a minimizer of $f$. Under these assumptions, \eqref{eq:SGD} where gradient estimators sample $\nabla f_i(x_k)$ uniformly (or via importance sampling) has been extensively analyzed \cite{bottou2018optimization,gower2019sgd,moulines2011non,needell2014stochastic,sebbouh2021almost,vaswani2019fast,pmlr-v37-zhaoa15}. In particular, when $\sigma_\star^2=0$ (e.g., in the finite-sum setting for overparameterized models, where $\sigma_\star^2=0$, see, e.g., \cite{ma2018power,vaswani2019fast}), \eqref{eq:SGD} with constant stepsizes achieves optimal $\mathcal{O}(1/k)$ and linear rates, as in the deterministic setting. In addition, if estimators of $\nabla f(x_k)$ are built using variance-reduction techniques \cite{defazio2014saga,gower2020variance,johnson2013accelerating,nguyen2017sarah,roux2012stochastic,schmidt2017minimizing,shalev2013stochastic,xiao2014proximal}, i.e., so that the variance of the noise on $\nabla f(x_k)$ naturally decreases iteratively, mitigating the effect of $\sigma_\star^2$, then \eqref{eq:SGD} achieves linear convergence on strongly convex functions. Examples include SAGA \cite{defazio2014saga} and SVRG \cite{johnson2013accelerating,xiao2014proximal}. In addition to \eqref{eq:SGD}, proximal point methods \cite{traore2024variance} and accelerated methods \cite{allen2018katyusha,fercoq2015accelerated,lin2015universal,vaswani2019fast,zhou2019direct} have also been analyzed in this setting.

Further, \eqref{eq:SGD} has been analyzed under weak and strong growth conditions, that is, respectively, for some $A,B\geq 0$, $\E_{\ve_k}[\|\varepsilon_k\|^2] \leq A(f(x_k)-f(x_\star))$ \cite{vaswani2019fast} and $\E_{\ve_k}[\|\varepsilon_k\|^2] \leq B\|\nabla f(x_k)\|^2$ \cite{schmidt2013fast}, where $x_\star$ is a minimizer of $f$ and $(x_k)_{k=1,2,\ldots}$ are the iterates of the method. These conditions yield optimal rates on convex and strongly convex functions. Those same convergence rates have been proven in the block-coordinate setting \cite{beck2013convergence,nesterov2012efficiency,wright2015coordinate}.

Finally, beyond smooth convex optimization, stochastic proximal methods have also been designed and analyzed on \emph{non-smooth} functions, see, e.g.,  \cite{duchi2009efficient,rosasco2020convergence}, and stochastic first-order methods on \emph{non-convex} functions have been, more recently, investigated under several assumptions, see, e.g., \cite{asi2019stochastic,danilova2022recent,davis2019stochastic,davis2019proximally,drori2020complexity,ghadimi2013stochastic,khaled2020better}.
\paragraph{Unified analysis of stochastic first-order methods}
As highlighted in, e.g., \cite{gorbunov2020unified,traore2024variance}, analyses of first-order methods across many settings can be unified by analyzing problem classes encoding (i) the method, (ii) the function class considered, and (ii) the expectation and variance of the noise $\varepsilon_k$. Such analyses yield convergence rates valid for all distributions satisfying the noise model. In particular, \cite[Assumption 4.1]{gorbunov2020unified} introduces the class of noise models for which there exist non-negative constants $A_1,A_2,C,D_1,D_2,\rho\in \R$ and a possibly random sequence $(\sigma_k^2)$ such that
\begin{align}\label{eq:assum_variance_gorbu}
    \E[\varepsilon_k | x_k] = 0, \quad \E[\|\varepsilon_k\|^2 | x_k] &\leq A_1(f(x_k) - f(x_\star))-\|\nabla f(x_k)\|^2 + C \sigma_k^2+D_1, \\
\E[\sigma_{k+1}^2 | x_k] &\leq (1 - \rho) \sigma_k^2 + A_2(f(x_k) - f(x_\star))+D_2,\nonumber
\end{align}
where $x_\star$ is a minimizer of $f$, $(x_k)_{k=1,2,\ldots}$ are the iterates of the method, and the expectation is taken with respect to the randomness of the method. Assumption \eqref{eq:assum_variance_gorbu} suffices for linear convergence of \eqref{eq:SGD}, on smooth strongly convex functions, to a neighborhood depending on $C_1,C_2$ \cite[Theorem 4.1]{gorbunov2020unified}, and this assumption is satisfied by most settings mentioned above, see \cite[Table 2]{gorbunov2020unified}. 

For instance, the finite-sum setting, where each \( f_i \) is convex and \( L_i \)-smooth, and gradient estimators uniformly sample \( \nabla f_i(x_k) \), satisfies \eqref{eq:assum_variance_gorbu} with \( A = 4 \max_i L_i \) and \( B = 2 \sigma_\star^2 \); see, e.g., \cite[Lemma 4.20]{garrigos2023handbook}. Similarly, the additive bounded noise model on \( L \)-smooth functions satisfies \eqref{eq:assum_variance_gorbu} with \( A = 2L \) and \( B = \sigma^2 \); see, e.g., \cite[Example 2]{stich2019unified}. These shared parameters explain the similarity in convergence rates across these noise models. Likewise, the strong growth condition, the additive relative noise model, and the block-coordinate setting can all be captured as a single noise model under this unified noise framework, as they share the same expectation and variance.
\paragraph{Computer-aided analysis of stochastic first-order methods}
Computer-aided techniques for automatically analyzing and designing first-order methods have recently attracted significant attention; see, e.g., \cite{taylor2024towards} for a recent survey. These include (i) the automatic computation of worst-case instances for a given method over a function class $\F$ after $N$ iterations, known as the Performance Estimation (PEP) framework \cite{drori2014performance,taylor2017smooth}; (ii) the automated construction of Lyapunov (or potential) functions to certify (sub)linear convergence by analyzing a small number of iterations \cite{lessard2016analysis,taylor2018lyapunov,taylor2019stochastic}; and (iii) constructive approaches for designing optimal or efficient methods tailored to specific function classes \cite{das2024branch,drori2020efficient,kim2016optimized,taylor2023optimal}. In a wide range of settings, the PEP framework provides tight \emph{numerical} worst-case guarantees as solutions of semidefinite programs (SDP). However, the size of the resulting SDPs increases with the number of iterations $N$, and extracting analytical closed-form expressions from the numerical results is generally a challenging task. On the other hand, Lyapunov certificates require initial guessing on their structure and may not yield tight guarantees, but are computationally and algebraically more tractable, since the associated SDPs are small-sized. Such computer-aided analyses are mainly restricted to deterministic settings, but there exist extensions to the stochastic setting that we quickly summarize.

First, computer-aided analyses naturally extend to the analysis of \emph{inexact} methods, where gradient errors are deterministic perturbations, as in \cite{doi:10.1137/060676386,devolder2013first,devolder2014first}, see, e.g., \cite{de2020worst,gannot2022frequency,hu2021analysis,taylor2017exact,vernimmen2024convergence,vernimmen2025empirical,vernimmen2025worst} for PEP analyses of inexact methods. We refer to this approach as the \emph{worst-scenario} approach. Indeed, conducting such analyses on settings in which perturbations are, in fact, stochastic rather than deterministic, yields convergence rates that hold for the worst realization of the noisy perturbations. Such rates are thus pessimistic in these settings.

Second, when the noise takes values in a finite set of cardinality $n$, e.g., in the finite-sum and the block-coordinate setting, stochastic methods can be analyzed \emph{exactly} by enumerating all scenarios and considering the average-case performance, see, e.g., \cite{cortild2025new,hu2017unified,kamri2023worst,taylor2019stochastic}. We refer to this approach by the \emph{all-scenarios} approach. The resulting numerical convergence rates are tight, and the associated closed-form analytical rates may remain valid for infinite-dimensional supports, as showcased, e.g., in \cite[Appendix E]{taylor2019stochastic}. On the downside, the size of the associated SDPs grows like $n^N$, limiting the scalability of this approach and restricting its use to the analysis of a small number of iterations only. Moreover, when rates depend on the support cardinality $n$, extracting closed-form expressions requires solving a sequence of increasingly large problems, whose solutions provide intuition for the analytical expression. 

Then, \cite{van2021speed} studies the additive bounded noise model by (i) theoretically linking SDP solutions to robustness and convergence guarantees (see, e.g., \cite[Theorems 5.3, 6.2]{van2021speed}), and (ii) leveraging these SDPs to analyze or design methods.

Finally, \cite{abbaszadehpeivasti2022convergence} analyzed the RCD by incorporating expected values directly as variables in the associated SDPs, under necessary constraints meant to ensure the variables are indeed expectations. See, also, \cite{drori2014performance} for an early introduction of this same idea, under the additive bounded noise model.

\subsection{Contribution}
Inspired by the idea of \cite{abbaszadehpeivasti2022convergence,drori2014performance} to embed expected values as variables to PEP formulations, we present a PEP-based approach for analyzing all stochastic zero-mean noises, uncorrelated across queries, and such that a possibly random sequence $(\sigma_k^2)_{k=0,1,\ldots,N-1}$, and for some constants $A_1,B_1,C_1,E_1,A_2,B_2,C_2,D_2,\rho \in \R$, it holds \begin{align}\begin{split}\label{eq:assum_noise}
   \E[\|\ve_k\|^2|\A_k]&\leq A_1(f(x_k)-f(x_\star))+B_1\|g_k\|^2+C_1\|x_k-x_\star\|^2+D_1+E_1\sigma^2_k,\\
      \E[\sigma_{k+1}^2|\A_k]&\leq A_2(f(x_k)-f(x_\star))+B_2\|g_k\|^2+C_2\|x_k-x_\star\|^2+D_2+(1-\rho)\sigma^2_k,\end{split}
\end{align}
where $(x_k)_{k=0,1,\ldots,N}$ is the sequence of iterates of a given first-order method, $g_k\in\partial f(x_k)$, $x_\star$ is a minimizer of $f$, and $\A_k:=\sigma(x_0,\ldots,x_k,\ve_0,\ldots,\ve_{k-1},\sigma_0^2,\ldots,\sigma_{k}^2)$ is the $\sigma$-algebra of available information up to iteration $k$, see, e.g., \cite{durrett2019probability}. This model includes, as special cases, all classical noise assumptions presented in Section \ref{sec:related_works}. 

More precisely, we 
\begin{itemize}
    \item Propose a framework, fitting the PEP methodology, for automated analysis of first-order stochastic methods with fixed stepsizes across a wide range of function classes (specified in Assumption \ref{assum:prob_def}), under all zero-mean noise models satisfying \eqref{eq:assum_noise} (Section~\ref{sec:framework}). Given $N$ a number of iterations, a method, a function class, and a class of noise models (expectation and variance), we propose a SDP of size $\propto N$, whose solution provides a convergence rate on the problem, which happens to be, in some cases, tight (Section \ref{sec:num_ill}).

    \item Use the proposed framework to analyze \eqref{eq:SGD} on the class of smooth (strongly) convex functions, under several noise assumptions. These assumptions cover the additive bounded noise model, the finite-sum setting with or without variance-reduction techniques, and the block-coordinate setting. We show both analytically (Theorem \ref{thm:absolute_noise_cst_stepsize}, Proposition \ref{prop:bound_A1_C1}) and numerically how the proposed framework fits in the existing first-order literature. In particular, it improves on all existing rates that only use the noise's expectation and variance, and are hence valid for any distribution satisfying these assumptions. In addition, the framework provides overall improved guarantees, i.e., even as compared to convergence rates derived on the basis of more information on the noise distribution than its variance and expectation, (i) on unstructured noise models as the additive bounded noise model, and (ii) in certain structural settings, on sublinear convergence rates (Section~\ref{sec:num_ill}). 
\end{itemize}
By construction, on \emph{structural} noise models, e.g., on the finite-sum or the block-coordinate settings, the all-scenarios approach introduced in \cite{taylor2019stochastic}, which uses this structural information on the noise's distribution, yields stronger convergence rates than the proposed framework. 
However, (i) these rates are valid for a smaller class of problems, preventing a unified analysis across different settings, (ii) the all-scenarios approach is restricted to finite-support noises, and can only improve on our framework on models for which information on the noise distribution is available, and (iii) this approach is mainly restricted, due to scalability issues, to a Lyapunov-based analysis. Hence, in some cases of sublinear convergence, our framework, by solving small-sized SDPs, improves on the all-scenarios approach, given the conservatism introduced by guessing the Lyapunov function's structure.
\subsection{Notation}
Throughout, we consider the standard Euclidean space $\R^d$ endowed with standard inner product $\langle\cdot,\cdot\rangle$ and induced norm $\|\cdot\|$. Given $N \in \mathbb{N}$, we let $[ N]=\{0, \cdots,N\}$. Given $k\in\N$, we denote by $\e_k$ a unit vector with $1$ in position $k$. We let $\Sym_+^d:=\{M\in \R^{d\times d}| \ M \text{ symmetric}, \ M\succeq 0\}$ be the cone of $d$-dimensional symmetric positive semidefinite matrices, and denote by $\D^d:=\{M\in \R^{d\times d}| \ M \text{ diagonal}\}$. In addition, given $M\in \R^{d\times d}$, we denote by $\|M\|=\max_{x\in \R^d:\ \|x\|=1}\ \|Mx\|$. 

\paragraph{Function classes}
Given a function $f:\R^d\to \R$, we say $f$ is convex if $f(\alpha x+(1-\alpha)y)\leq \alpha f(x)+(1-\alpha)f(y) \ \forall x,y \in \R^d, \ \alpha \in (0,1)$. For some parameter $\mu\geq 0$, we say $f$ is $\mu$-strongly convex if $f-\frac{\mu}{2}\|\cdot\|$ is convex.  
We refer by $\partial f(x)$ to the subdifferential \cite{rockafellar1997convex} of $f$ at $x$, defined as the set of vectors $v_x\in \R^d$ satisfying $f(y)\geq f(x)+\langle v_x,y-x\rangle$, and denote by $v_x$ a gradient of $f$ at $x$ (i.e., $v_x\in\partial f(x)$). 
In addition, for some $L\geq 0$, we say $f$ is $L$-Lipschitz continuous if $\|f(x)-f(y)\|\leq L\|x-y\|, \ \forall x,y\in \R^d $, and $f$ is $L$-smooth if it is differentiable everywhere with first derivative $L$-Lipschitz continuous. Finally, we denote by $x_\star$ a minimizer of $f$.

We denote by $\F_{\mu,L}$ the class of $\mu$-strongly convex $L$-smooth functions (with $\mu=0$ for convex functions), with convention $L=\infty$ for $\mu$-strongly convex functions that are not smooth. 
\paragraph{Random variables}
Given $N$ random variables $\ve_0,\ldots,\ve_{N-1}$, we denote by $\ve$ the random variable that aggregates all $\ve_k$' s, whose distribution depends on those of $\ve_k$' s, i.e. $\ve=[\ve_0,\ldots,\ve_{N-1}]$. We denote by $\E[\cdot]$ the expectation with respect to the randomness of the method under analysis, that is, the expectation taken over the joint distribution of all random variables involved in the method. 

Finally, we let $\{\A_k\}_{k\in [N]}$ be the filtration of available information at iteration $k$, i.e., 
$\A_k=\sigma(x_0,\ldots,x_k,\ve_0,\ldots,\ve_{k-1},\sigma_0^2,\ldots,\sigma_{k}^2)$.
\section{Preliminaries: computer-aided analysis of inexact methods.}
This section introduces the core concepts of the Performance Estimation Problem (PEP) framework, originally proposed in~\cite{drori2014performance} and further developed in~\cite{taylor2017smooth}, through a simple illustrative example, as in, e.g., \cite[Section 2]{taylor2024towards}. Consider a single iteration of the inexact gradient descent on the class $\mathcal{F}_{\mu,L}$ of $L$-smooth ($\mu$-strongly) convex functions, under additive deterministic perturbation of bounded norm $|\sigma|$. Given $x_\star$ a minimizer of $f$, let the goal be to derive the worst-case upper bound on $f(x_1) - f(x_{\star})$, valid for all functions in $\mathcal{F}_{\mu,L}$, under the assumption that $\|x_0 - x_\star\|^2 \leq 1$. The tightest possible such worst-case bound can be formulated as the solution to the following optimization problem:
\begin{align}    \label{eq:PEP_untractable_preliminaries}
    \underset{\substack{d\in\N,\ f \in \mathcal{F}_{\mu,L},\\ x_0,x_\star\in\R^d,\\ \varepsilon_0\in \R^d}}{\max} \bigg\{ f(x_1) - f(x_\star) \ \big|\ 
    & \|x_0 - x_\star\|^2 \leq 1,  x_1 = x_0 - \alpha \left( \nabla f(x_0) + \varepsilon_0 \right), \nonumber\\&\nabla f(x_\star) = 0,\  
    \|\varepsilon_0\|^2 \leq \sigma^2 \bigg\}.
\end{align}
\paragraph{Finite-dimensional reformulation of (\ref{eq:PEP_untractable_preliminaries})}
As such, given the infinite-dimensional variable $f\in\F_{\mu,L}$, \eqref{eq:PEP_untractable_preliminaries} is infinite-dimensional and untractable. However, as shown, e.g., in \cite{drori2014performance,taylor2017smooth}, it can be equivalently reformulated as a maximization over finite-dimensional variables. Typically, consider a set $S=\{(x_i,f_i,g_i)\}_{i=0,1,\star}\in(\R^d\times \R\times \R^d)^3$, meant to encode the iterates of \eqref{eq:SGD} and evaluations of the function and (sub)gradient at these iterates, and $\ve_0\in\R^d$. To obtain an equivalent reformulation of \eqref{eq:PEP_untractable_preliminaries} as a maximization over $S\cup\ve_0$, one should ensure (i) that $S\cup\ve_0$ is consistent with the dynamics of the method and the problem assumptions, i.e., it satisfies \eqref{eq:SGD}, $\|x_0-x_\star\|^2\leq 1$, $g_\star=0$, and $\|\ve_0\|^2\leq\sigma^2$, and (ii) that $S$ indeed contains evaluations of a function in $\F_{\mu,L}$, i.e. that $S$ is \emph{$\F_{\mu,L}$-interpolable}.
\begin{definition}[$\F$-interpolability]\label{def:F_interpolability}
Given $N\in \mathbb{N}$, consider a function class $\F$ and a finite set $S=\{(x_i,f_i,g_i)\}_{i\in \NN}\subset (\R^d\times \R\times \R^d)^{N+1}$. We say $S$ is $\F$-interpolable if there exists some $f\in \F$ such that
\begin{align}\label{eq:F_interpolability}
    f(x_i)=f_i, \text{ and } g_i\in \partial f(x_i), \ \forall i\in\NN.
\end{align}
\end{definition}

For some function classes $\F$, $\mathcal{F}$-interpolability of a set $S$ can be expressed via a set of algebraic conditions involving the elements of $S$. When such conditions are both necessary and sufficient for $\F$-interpolability, they are referred to as \emph{interpolation conditions}, and they enable an exact reformulation of~\eqref{eq:PEP_untractable_preliminaries} as a finite-dimensional optimization problem. On the other hand, if only necessary (but not sufficient) conditions for $\mathcal{F}$-interpolability are imposed, the resulting optimization problem becomes a relaxation of~\eqref{eq:PEP_untractable_preliminaries}, which may yield upper bounds on the true worst-case performance. For $\F_{\mu,L}$, interpolation conditions are known, see \cite{taylor2017smooth} (see also \cite{azagra2017whitney}).
\begin{proposition}[Interpolation conditions for $\F_{\mu,L}$, \cite{taylor2017smooth}, Theorem 4]\label{prop:FmuL_CI}
Given $-\infty<\mu<L\leq +\infty$ and $N\in \mathbb{N}$, consider $\F_{\mu,L}$ and a set $S=\{(x_i,f_i,g_i)\}_{i\in \NN}\subset (\R^d\times \R\times \R^d)^{N+1}$. Then, $S$ is $\F$-interpolable if and only if, $\forall i,j \in \NN$,
\begin{align}\label{eq:FmuL_CI}
    f_j\geq f_i+\langle g_i,x_j-x_i\rangle +\frac{L}{2(L-\mu)}\bigg(\frac{\|g_i-g_j\|^2}{L}+\mu\|x_i-x_j\|^2-\frac{2\mu}{L}\langle g_i-g_j,x_i-x_j\rangle\bigg).
\end{align}    
\end{proposition}

Hence, \eqref{eq:PEP_untractable_preliminaries} can be equivalently reformulated as
\begin{align}\label{eq:PEP_finite_dim_preliminaries}
    \underset{\substack{S=\{(x_i,f_i,g_i)\}_{i=0,1,\star}\\\ve_0\in\R^d\\d\in\N}}{\max} \bigg\{f_1-f_\star|\ &S \text{ satisfies \eqref{eq:FmuL_CI}}, \ \|x_0-x_\star\|^2\leq 1,\nonumber\\& x_1=x_0-\alpha (g_0+\varepsilon_0), \ g_\star=0,\ \|\ve_0\|^2\leq \sigma^2\bigg\}.
\end{align}
\paragraph{SDP expression of (\ref{eq:PEP_untractable_preliminaries})}
While problem~\eqref{eq:PEP_finite_dim_preliminaries} is finite-dimensional, it remains difficult to solve due to its nonconvex nature, which stems from the scalar products between iterates and (sub)gradients that appear in the interpolation conditions~\eqref{eq:FmuL_CI}. To overcome this difficulty, \cite{drori2014performance,taylor2017smooth} proposed an alternative parametrization. Rather than optimizing over the set of variables $S \cup \{\varepsilon_0\}$ directly, they introduce a representation based on 
\begin{align}\label{eq:F,G}
    &\text{(i) the vector of function value } F = [ f_\star,\ f_0, \ f_1]^\top,\\
    &\text{(ii) the \emph{Gram matrix} }G = P P^\top, \text{ where } P = \begin{bmatrix} x_\star ,\  x_0,\ x_1,\ g_\star, \ g_0 ,\ g_1 ,\ \varepsilon_0 \end{bmatrix}^\top.\nonumber
\end{align}
By construction, $G$ captures all scalar products between vectors in $P$. This leads to an equivalent reformulation of~\eqref{eq:PEP_finite_dim_preliminaries} as a SDP:
\begin{align}\label{eq:PEP_SDP_preliminaries}
    \underset{\substack{F \in \mathbb{R}^2,\\ G \in \mathbb{S}_+^6}}{\text{max}} 
   \bigg\{ 
        F^\top b_{\text{obj}}
        |\ 
        & F^\top b_\text{interp}^{ij} + \operatorname{Tr}(A_{\text{interp}}^{ij} G) \leq 0,\ \forall i,j \in \{0,1,\star\},\  \operatorname{Tr}(A_\text{init} G) \leq 1,\nonumber\\&
          \operatorname{Tr}(A_\text{opt} G) = 0, \Tr(A_{\text{method}}G)=0, \ 
          \operatorname{Tr}(A_{\varepsilon_0} G) \leq \sigma^2
    \bigg\},
\end{align}
for some matrices $A_{\text{interp}}^{ij}, A_\text{init}, A_\text{opt}, A_{\varepsilon_0}, A_{\text{method}} \in \mathbb{S}^6$ and vectors $b_{\mathrm{obj}}, b_\text{interp}^{ij}\in \mathbb{R}^3$ that encode the objective and constraints of \eqref{eq:PEP_finite_dim_preliminaries}. For instance, $b_{\text{obj}}=\e_3-\e_1$, where $\e_k\in\R^3$ is the $k^{\th}$ unit vector, hence $ F^\top b_{\text{obj}}=f_1-f_\star$, and $A_{\text{init}}=\begin{pmatrix}
    1 &-1&0&0&0&0\\
    -1 &1&0&0&0&0\\
    0 &0&0&0&0&0\\
    0 &0&0&0&0&0\\
    0 &0&0&0&0&0
\end{pmatrix}$, hence $\operatorname{Tr}(A_\text{init} G) =\|x_0-x_\star\|^2$. Maximization over the dimension $d$ is automatically satisfied, since \eqref{eq:PEP_SDP_preliminaries} is dimension-independent, hence valid for all $d$. 

Although introduced here on a simple example, the PEP framework can be generalized to obtain tight worst-case guarantees for a wide variety of optimization problems. These guarantees are obtained by solving SDPs, provided that the ingredients of the analysis—namely, the algorithm, the function class, the performance criterion, and the initial conditions—can be encoded via \emph{Gram-representable} expressions, that is, expressions that depend linearly (i) on function values and (ii) on scalar products between iterates, (sub)gradients, and additive perturbations, that include $\ve$ and $\sigma^2$, as in \eqref{eq:assum_noise}.

\begin{definition}[Gram-representable expression]\label{def:gram}
Let $N \in \mathbb{N}$, and finite sets 
$S = \{(x_i, f_i, g_i)\}_{i \in \NN\cup\star} \subset (\mathbb{R}^d \times \mathbb{R} \times \mathbb{R}^d)^{N+2}$, and $V=\{(\ve_i,\sigma^2_i)\}_{i \in[N-1]}\in(\R^d\times \R^d)^N$. An algebraic expression $p$ involving the elements of $S\cup V$ is said to be \emph{Gram-representable} if it can be written using a finite number of functions, linear in the $f_i$'s, $\sigma^2_j$'s and scalar products in $x_i, g_i, \ve_j$, $i\in[N]\cup\star$, $j\in[N-1]$, that is
\begin{align*}
&\langle g_i, g_j \rangle, \ \langle x_i, x_j \rangle, \ \langle g_i, x_j \rangle, \ \text{for all } i, j \in \NN\cup\star, \\& \langle g_i, \ve_j \rangle, \ \langle x_i, \ve_j \rangle, \ \text{for all } i\in \NN\cup\star, \ j\in[N-1], \text{ and } \\&\langle \ve_i, \ve_j \rangle,  \text{ for all } i, j \in [N-1].
\end{align*}
\end{definition}

Given a Gram-representable expression $p$, we say $(F,G)$ satisfies $p$, if it is the Gram representation \begin{align} \label{eq:equivFGSV}      
    F&:=[f_\star,\ f_0,\ldots,f_N,\ \sigma^2_0,\ldots,\sigma^2_{N-1}]^\top,\\
    G&:=PP^\top,\text{ where } P:=[x_\star,\ x_0,\ldots,x_N,\ g_\star,\ g_0,\ldots,g_N,\ve_0,\ldots,\ve_{N-1}]^\top.\nonumber
\end{align}
of a set $S\cup V$ satisfying $p$. For instance, given $k\in [N-1]$, $(F,G)$ satisfies $\|\ve_k\|^2\leq \sigma^2$ if $\vve_k^\top G\vve_k\leq \sigma^2$, where $\vve_k=\e_{2N+5+k}$.
\section{Framework for computer-aided analysis of stochastic methods} \label{sec:framework}
This section presents a framework for automated analysis of stochastic first-order methods; via the introduction of expected values as variables to PEP formulations. We first define the classes of problems on which the framework applies, before presenting the associated SDP formulation.
\subsection{Problem definition}\label{sec:prob_def}
Given $N\in \mathbb{N}$, we analyze $N$ iterations of a method $\M$ on a function class $\F$, under stochastic noise drawn from an element of a class of distributions $\Omega_N$. Performance of $\M$ is evaluated through a measure $\perf$, under an initial condition $\Ci$. We impose a number of assumptions on these objects.
\begin{assum} \label{assum:prob_def} Let $N\in \mathbb{N}$, a method $\M$, a function class $\F$, a performance measure $\perf$, a class of distributions $\Omega_N(A_1,A_2,B_1,B_2,C_1,C_2,D_1,D_2,E_1,\rho)$, and an initial condition $\Ci$. We assume the following.
    \begin{itemize}
    \item \textbf{Method $\M$} \cite[Definition 2.11]{taylor2017exact}. We consider stochastic fixed-step linear first-order methods (FSLFOM), whose $k^{\th}$ iteration on a function $f$ is given by 
    \begin{align}\label{eq:method}\tag{$\mathcal{M}$}
        x_{k+1}=x_0+\sum_{j=0}^{k}\alpha_{k,j}(g_j+\ve_j),
    \end{align}
where $g_j\in\partial f(x_j)$, $\ve_j$ is a random variable, and all coefficients $\alpha_{k,j}$ are predetermined.

Denote by $S=\{(x_j,f(x_j),g_j)\}_{j\in \NN\cup\star}$ the iterates $(x_j)_{j\in \NN}$ of $\M$, and a minimizer $x_\star$ of $f$, with associated function and (sub)gradient values.
\item \textbf{Distribution class $\Omega_N(\mathbf{P})$}. Letting $\mathbf{P}=A_1,A_2,B_1,B_2,C_1,C_2,D_1,D_2,E_1,\rho$, we consider classes of distributions $\Omega_N(\mathbf{P})$ such that any random variable $(\ve,\sigma^2)=\{(\ve_i,\sigma^2_i)\}_{i\in \lb N-1\rb}$ drawn from a distribution in $\Omega_N(\mathbf{P})$ satisfies

\hspace{0.5cm} (i) $\E[\varepsilon_k|\A_k]=0$, $\forall k\in \NN$, where $\A_k$ is the filtration of available information at iteration $k$, and

\hspace{0.5cm} (ii) $(\varepsilon, \sigma^2)$ satisfies \eqref{eq:assum_noise} w.r.t. $S$, with parameters $\mathbf{P}$. 
\item \textbf{Function class $\F$} \cite[Definition 2.2]{taylor2017exact}. We consider function classes whose interpolation conditions, see Definition \ref{def:F_interpolability}, are Gram-representable on realizations of $S$ (Definition \ref{def:gram}).
\item \textbf{Performance measure $\mathcal{P}$, and initial condition $C_{\text{init}}$} \cite[Definitions 2.3,2.4]{taylor2017exact}. We consider performance measures and initial conditions that are Gram-representable on realizations of $S$ (Definition \ref{def:gram}).\end{itemize}
\end{assum}

In what follows, we remove dependence of $\Omega_N$ in $\mathbf{P}$, to lighten notation. By an abuse of notation, we denote by $(\ve,\sigma^2)\in\Omega_N$ any pair of random variables drawn from a distribution in $\Omega_N$, whatever this specific distribution is. Observe that Assumption \ref{assum:prob_def} allows the noise to follow different distributions from query to query, i.e., $\ve_k$'s distribution may differ from that of $\ve_l$, when $l\neq k$.

The proposed framework adapts to more general settings than those satisfying Assumption \ref{assum:prob_def}, e.g., there could be more randomness than simply noisy (sub)gradient evaluations. However, Assumption \ref{assum:prob_def} simplifies the exposition of the framework, and already encompasses a broad range of settings, including, e.g., stochastic (accelerated) gradient descent on $\F_{\mu,L}$, where the initial condition and performance measure may depend on $\|x_0 - x_\star\|^2$, $\|g_0\|^2$ for $g_0 \in \partial f(x_0)$, or $f(x_0) - f(x_\star)$, with analogous expressions involving $x_N$, $g_N$, and $f(x_N)$ in the performance criterion. 

\paragraph{Central question} We seek to obtain, given satisfaction of $\Ci$, the tightest bound possible on the expected value of the performance measure $\perf$, which depends on $f$, and the iterates $(x_k)_{k\in[N\cup \star]}$ of $\M$; valid for any $f\in \F$ and $(\ve,\sigma^2) \in \Omega_N$, that is,
\begin{align}\label{eq:PEP_untractable_gen}
    \underset{\substack{d\in \N,\\f\in \F\\x_0,x_\star \in\R^d\\(\ve,\sigma^2)\in \Omega_N}}{\max} \bigg\{\E[\perf]|\ &\Ci \text{ is satisfied },  (x_k)_{k\in \NN} \text{ is generated by $\M$ w.r.t. $\ve$},\nonumber\\& 0\in \partial f(x_\star)\bigg\}.
\end{align}
Note that $x_0$ and $x_\star$ (and the associated function and (sub)gradient values) are deterministic, in contrast to $x_k$, $k\geq 1$ (and associated function and (sub)gradient values), which are stochastic. 

As an example, consider the analysis of a single iteration of \eqref{eq:SGD} on $\F_{\mu,L}$, under the additive bounded noise model.
\begin{example}[SGD under additive bounded noise model]\label{ex:running_ex}
    Let $N=1$, $\M$ be \eqref{eq:SGD}, $\F=\F_{\mu,L}$ for some $-\infty<\mu<L\leq +\infty$, $\perf=f(x_1)-f(x_\star)$, $\Ci=\|x_0-x_\star\|^2\leq 1$, and all parameters in \eqref{eq:assum_noise} be zero, except for $D_1=\sigma^2$, for some $\sigma^2\in \R$. Then, \eqref{eq:PEP_untractable_gen} amounts to \begin{align}\label{eq:PEP_untractable}
    \underset{\substack{d\in \N,\\f\in \F_{\mu,L}\\ x_0,x_\star\in\R^d,\\ \ve_0 \text{ a r.v.}}}{\text{ max }} \bigg\{\E[f(x_1)-f(x_\star)]|\ & \|x_0-x_\star\|^2\leq 1,\ x_1=x_0-\alpha (\nabla f(x_0)+\varepsilon_0),\\ & \nabla f(x_\star)=0,\ \E[\ve_0]=0, \ \E[\ve_0^2]\leq \sigma^2\bigg\}.\nonumber
\end{align}
\end{example}
\subsection{SDP relaxation of (\ref{eq:PEP_untractable_gen})}\label{sec:SDP_rel}
\paragraph{Exact reformulation of \eqref{eq:PEP_untractable_gen}}
As in the inexact or deterministic case, \eqref{eq:PEP_untractable_gen} is as such untractable, both given the infinite-dimensional variable $f\in\F$, the random variables $(\ve,\sigma^2)\in \Omega_N$, and the expectation $\E[\perf]$. We thus introduce sets of finite-dimensional variables, along with corresponding constraints that ensure maximizing over these variables is equivalent (or provides a relaxation as tight as possible) to the original problem \eqref{eq:PEP_untractable_gen}. In particular, given $f\in\F$, $x_\star\in \R^d$ a minimizer of $f$, $x_0\in \R^d$ a starting point satisfying the initial condition $\Ci$, $(\ve,\sigma^2)\in\Omega_N$, and $(x_k)_{k\in\NN}$ the set of iterates generated by $\M$ accordingly, we define
\begin{align}\label{eq:ScupV}
    S\cup V, \text{ where }& S=\{(x_j,f(x_j),g_j)\}_{j\in\NN\cup\star}, \ g_j\in \partial f(x_j), \ \forall j\in\NN\cup\star,\nonumber\\\text{ and }& V=\{(\ve_j,\sigma^2_j)\}_{j\in \lb N-1\rb}.
\end{align}
Further, under Assumption \ref{assum:prob_def}, relevant quantities in \eqref{eq:PEP_untractable_gen} are Gram-representable, hence it is convenient to consider the Gram representation of $S\cup V$, i.e., 
\begin{align}   \label{eq:FG}        
    F&:=[f(x_\star),f(x_0),\ldots,f(x_N),\ \sigma_0^2,\ldots,\sigma_{N-1}^2]^\top,\\
    G&:=PP^\top,\text{ where } P:=[x_\star,x_0,\ldots,x_N,g_\star,g_0,\ldots,g_N,\ve_0,\ldots,\ve_{N-1}]^\top.\nonumber
\end{align}
One cannot directly optimize over the stochastic quantity $(F,G)$, or over all its possible realizations, as done in \cite{kamri2023worst,taylor2019stochastic,cortild2025new}, since there might be an infinite number of them. Hence, following the approach of \cite{abbaszadehpeivasti2022convergence}, we consider \emph{expectations} of $(F,G)$ as variables of interest. Specifically, let 
\begin{align}
    (F_\E,G_\E):=&(\E[F],\E[G]),
\end{align}
that thus depend on $f,x_\star,x_0,\M$, and $(\ve,\sigma^2)$. $\EF$ contains, as elements, $\E[f(x_i)]$, and $\EG$ contains, e.g., $\E[\langle x_i,g_j\rangle]$, $\E[\langle x_i,\ve_j\rangle]$, ...,

Our goal is to maximize the expectation of the performance criterion $\perf$, which is, under Assumption \ref{assum:prob_def}, linear in realizations of $(F,G)$. Equivalently, by linearity of the expectation, it suffices to maximize the performance criterion over all possible $(F_\E,G_\E)$, generated by any $f\in\F$, $(\epsilon,\sigma)\in\Omega_N$, $x_0$ and $x_\star$. We therefore introduce the set $\Ge(\M,\F,\Omega_N,\Ci)$ of pairs $(F,G)$ that correspond to some $(\EF,\EG)$, i.e.,  
\begin{equation}\label{eq:GE}
    \Ge:= \left\{ 
    \begin{aligned}
        (F,G)\in \R^{2N+2}\times \Sym_+^{3N+4}:\ & \exists \ f\in \F,\ x_\star \in\R^d \text{ a minimizer of }f, \\ & x_0\in\R^d \text{ satisfying }\Ci,\ (\ve,\sigma^2)  \in \Omega_N,\\& \text{ and } (x_k)_{k\in\NN} \text{ the iterates of }\M
        \\\text{ s.t. } &\ (F,G)= (F_\E,G_\E)
    \end{aligned} 
    \right\},
\end{equation}
where the dependence of $\Ge$ on $\M,\F,\Omega_N$ and $\Ci$ is removed and supposed clear from the context. Then, \eqref{eq:PEP_untractable_gen} can be reformulated as
\begin{align}\label{eq:maxGE}\tag{$P_{\text{exact}}$}
    \max_{(F,G)\in \G_E} \Tr(A_{\text{obj}}G)+F^\top b_{\text{opt}},
\end{align}
where the linear terms \(\Tr(A_{\mathrm{obj}} G)\) and \(F^\top b_{\mathrm{opt}}\) encode the performance measure \(\perf\). To solve \eqref{eq:maxGE}, our goal is to represent $\Ge$, using a finite number of algebraic conditions. As in \cite{abbaszadehpeivasti2022convergence}, we first introduce conditions that are \emph{necessary}, but not sufficient, for inclusion in $\G_E$.

As a reminder, given a Gram representable condition $p$, we say $(F,G)$ satisfies $p$ if it is the Gram representation \eqref{eq:equivFGSV} of some $S=\{(x_i,f_i,g_i)\}_{i\in \NN\cup\star}$ and $V=\{(\ve_i,\sigma^2_i)\}_{i\in \lb N-1 \rb}$ satisfying $p$, see Appendix \ref{app:fullLMI} for details. The necessary conditions describing $\Ge$, given $\M,\F,\Omega_N$ and $\Ci$ satisfying Assumption \ref{assum:prob_def}, follow from
\begin{itemize}
    \item linear constraints that necessarily hold on any realization, imposed in expectation—namely, any $(F,G)\in\Ge$ should be \(\F\)-interpolable, satisfy the initial condition $\Ci$, the method definition \eqref{eq:method}, and the optimality condition $\|g_\star\|^2$; and
    \item constraints that naturally involve expectations and are Gram-representable. Specifically, these include (i) the noise variance ~\eqref{eq:assum_noise}, that is, $(F,G)\in\Ge$ should satisfy, $\forall k\in[N-1]$, and given parameters $A_1,A_2,B_1,B_2,C_1,C_2,D_1,D_2,\rho\in\R$,
    \begin{align}\tag{$\Cn$}\label{eq:Cn}
   \|\ve_k\|^2&\leq A_1(f_k-f_\star)+B_1\|g_k\|^2+C_1\|x_k-x_\star\|^2+D_1+E_1\sigma^2_k,\\
      \sigma_{k+1}^2&\leq A_2(f_k-f_\star)+B_2\|g_k\|^2+C_2\|x_k-x_\star\|^2+D_2+(1-\rho)\sigma^2_k,\nonumber
\end{align} 
as well as (ii) uncorrelation conditions involving \(\ve_k\) and \(\ve_l\) ($l\neq k)$, and $\ve_k$ and \(x_i\), \(g_i\) ($i\leq k$). Specifically, $(F,G)\in\Ge$ should satisfy
    \begin{align}\tag{$\Cind$}\label{eq:cind}
        &\langle \ve_k,\ve_l\rangle=\langle \ve_k,x_i\rangle=\langle \ve_k,g_i\rangle=0, \ \forall k \in \lb N-1 \rb,  i\ \leq k, l\neq k.
    \end{align}
\end{itemize}

For instance, letting $\vve_k=\e_{2N+5+k}$ be the selection vector accessing the element $\ve_k$ in $G$, for any $k\in[N-1]$, condition $\langle \ve_k, \ve_l\rangle=0$ writes $\vve_k^\top G \vve_l=0$. Hence, if $G=\EG$, it becomes $\E[\langle \ve_k,\ve_l\rangle]=0$, which holds by definition of $\Omega_N$. We now formalize the proposed conditions.
\begin{proposition}\label{prop:def_Gc}
   Given $N\in \N$, consider a method $\M$, a function class $\F$, a class of distributions $\Omega_N$, and an initial condition $\Ci$, satisfying Assumption \ref{assum:prob_def}. In addition, let $\G_E$ be defined in \eqref{eq:GE}. Then, any $(F,G)\in \Ge$ is $\F$-interpolable, and satisfies $\Ci$, $\Cn$, $\Cind$, $\|g_\star\|^2=0$, and \eqref{eq:method}.
\end{proposition}
\begin{proof}
    Let $(\EF,\EG)\in \Ge$. Denote by $(\tilde F, \tilde G)$ the stochastic pair for which $F=\E[\tilde F]$, $G=\E[\tilde G]$.

By definition of $\Ge$, given in \eqref{eq:GE}, there exist some $(\ve,\sigma^2)  \in \Omega_N,\ f\in \F$, $ x_\star \in\R^d \text{ minimizer of }f$, $x_0\in\R^d \text{ satisfying }\Ci$, and $(x_k)_{k\in\NN}$ the associated iterates of $\M$, such that $(\tilde F, \tilde G)$ is the Gram representation \eqref{eq:FG} of \break $S=\{(x_j,f(x_j),g_j)\}_{j\in \NN\cup \star}$, where $g_j\in\partial f(x_j)$, and $V=\{(\ve_j,\sigma_j^2)\}_{j\in \lb N-1 \rb}$. Hence, any realization of $(\tilde F,\tilde G)$ is $\F$-interpolable, and satisfies $\Ci$, $\|g_\star\|^2=0$ and \eqref{eq:method}. By linearity of these conditions in the realizations of $(\tilde F,\tilde G)$, and linearity of an expectation, it therefore holds that $(\EF,\EG)$ is $\F$-interpolable and satisfies $\Ci$. 

Further, satisfaction of $\Cn$ by $(\EF,\EG)$ follows straightforwardly from satisfaction of \eqref{eq:assum_noise} by $(\ve,\sigma^2)\in\Omega_N$. In addition, $\langle \ve_k,\ve_l\rangle=0$, $k\neq l$, and $\langle \ve_k,x_i\rangle=\langle \ve_k,g_i\rangle=0$, $i\leq k$, follows from $\E[\ve_k|\A_k]=0$. Indeed, then, e.g., by the law of total expectation, $\E[\langle \ve_k,x_i\rangle]=\E[\E[\langle \ve_k,x_i\rangle|\A_k]]=\E[\langle \E[\ve_k|\A_k],x_i\rangle]=0$, and the same holds for the other variables in $\A_k$.
\end{proof}

The necessary conditions introduced in Proposition \ref{prop:def_Gc} differ from the ones proposed in \cite{abbaszadehpeivasti2022convergence}, under the coordinate descent setting, in the condition $\Cind$. Indeed, in \cite{abbaszadehpeivasti2022convergence}, the only uncorrelation conditions imposed are $\langle \ve_k,x_k\rangle=\langle \ve_k,g_k\rangle=0$, $k\in [N-1]$.
\paragraph{SDP relaxation of (\ref{eq:PEP_untractable_gen})}
Relying on Proposition \ref{prop:def_Gc}, we introduce a set $\Gc$ satisfying $\G_E\subseteq\G_c$. 
\begin{align}\label{eq:GC}
        \G_c:= \bigg\{    (F,G)\in \R^{2N+2}\times \Sym_+^{3N+4}:\  &(F,G) \text{ is $\F$-interpolable, and satisfies } \\&\text{$\Ci$, \eqref{eq:method}, $\|g_\star\|^2=0$},\ \text{$\Cind$ and $\Cn$} \bigg\}.\nonumber
\end{align}
Under Assumption \ref{assum:prob_def}, this set is convex, since characterized by a set of linear conditions on $(F,G)$. In addition, it allows for a SDP relaxation of \eqref{eq:PEP_untractable_gen}, given by 
\begin{align}\label{eq:relax}\tag{$P_{\text{relax}}$}
    \max_{(F,G)\in \G_c} \Tr(A_{\text{obj}}G)+F^\top b_{\text{opt}},
\end{align}
where the linear terms \(\Tr(A_{\mathrm{obj}} G)\) and \(F^\top b_{\mathrm{opt}}\) encode the performance measure \(\perf\). Indeed, by Proposition \ref{prop:def_Gc}, $\G_E\subseteq \Gc$, hence the bound resulting from \eqref{eq:maxGE} is smaller than the one resulting from \eqref{eq:relax}.
\begin{example}
    Consider the setting of Example \ref{ex:running_ex}, i.e., a single iteration of \eqref{eq:SGD} on $\F_{\mu,L}$, under additive bounded noise. Then, \eqref{eq:relax} can be expressed as
\begin{align}\label{eq:PEP_SDP_ex}
    \underset{\substack{F \in \mathbb{R}^3,\\ G \in \mathbb{S}_+^6}}{\max} 
    \bigg\{ 
        F^\top b_{\mathrm{obj}} 
        |\ 
        & F^\top b_\text{interp}^{ij} + \operatorname{Tr}(A_{\text{interp}}^{ij} G) \leq 0,\ \forall i,j \in \{0,1,\star\},\  \operatorname{Tr}(A_\text{init} G) \leq 1,\\&
          \operatorname{Tr}(A_\text{opt} G) = 0,\ \Tr(A_{\text{method}}G)=0, \nonumber\\& 
          \operatorname{Tr}(A_{\varepsilon_0} G) \leq \sigma^2, \  \Tr(A_{\ve_0x_0}G)=\Tr(A_{\ve_0g_0}G)=0
    \bigg\},\nonumber
\end{align}
where $A_{ij}, A_0, A_\star, A_{\varepsilon_0} \in \mathbb{S}^6$, and $b_{\mathrm{obj}}, b_{ij} \in \mathbb{R}^3$ are defined as in \eqref{eq:PEP_SDP_preliminaries}, and \break $A_{\ve_0x_0}, A_{\ve_kg_0}$ encode $\Cind$. An explicit formulation of \eqref{eq:PEP_SDP_ex} is provided in Appendix \ref{app:fullLMI}.

Observe that ~\eqref{eq:PEP_SDP_ex} is equivalent to~\eqref{eq:PEP_SDP_preliminaries}, except for the independence conditions. The key difference lies in the interpretation of the variables: while~\eqref{eq:PEP_SDP_preliminaries} operates on the (deterministic) Gram representation of \(S \cup V\), the variables \((F,G)\) in~\eqref{eq:PEP_SDP_ex} represent the \emph{expectation} of this Gram representation. As a result, even when relaxing the exact stochastic formulation \eqref{eq:PEP_untractable_gen}, working with expectations directly allows for potentially tighter worst-case bounds than those derived from deterministic analyses.
\end{example}
\begin{remark}
    Although the variables $(\ve_k, \sigma_k^2)$ represent stochastic noise, $\G_c$ interprets them as deterministic entities satisfying \eqref{eq:assum_noise}, and orthogonal to all prior variables $x_i,g_i$, where $i\leq k$, and $\ve_i$, $i<k$.
\end{remark}

Elements $(G, F) \in \G_c$ are intended to represent expectations of the Gram representation of the set $S \cup V$ as defined in~\eqref{eq:ScupV}, but $\G_c$ only encodes necessary conditions for this to hold. In particular, $\G_c$ may include pairs $(G, F)$ that do not correspond to any valid expected Gram representation. This situation is analogous to the case where only necessary conditions for $\F$-interpolability are imposed, leading to sets that may contain elements not consistent with any function in $\F$. As an example, we present some of the potential sources of relaxation in \eqref{eq:relax}, arising from the difference between $\G_E$ and $\Gc$.
\paragraph{Sources of relaxation in \eqref{eq:relax}}
\begin{description}
    \item[Compensation across realizations.] In $\G_E$, the constraints $\Ci$, $\|g_\star\|=0$, \eqref{eq:method}, and $\F$-interpolability are satisfied by the union of all realizations of $(F,G)$, generated by a noise realization $\ve^{(i)}$. By contrast, in $\Gc$, these conditions are only required to hold in expectation, allowing for compensation across different realizations. As a result, some realizations may individually violate one of these constraints, while their average satisfies them.
    
    \item[Realizations interpolated by different functions.] Even if each realization of $(F,G)$ is $\F$-interpolable, they may be consistent with different functions in $\F$, rather than a single common one. Belonging to $\G_E$ requires that a single function $f \in \F$ simultaneously interpolates all realizations.
\end{description}
\section{Applications: performance analysis of SGD on $\F_{\mu,L}$}\label{sec:num_ill}
This section relies on the framework introduced in Section \ref{sec:framework} to analyze \eqref{eq:SGD} on $\mathcal{F}_{\mu,L}$, under several noise models. We compare the resulting convergence rates with existing guarantees, obtained classically or in a computer-aided way. We demonstrate both numerically (Figures \ref{fig:absolute_noise_cst_stepsize}, \ref{fig:sum_of_functions}, \ref{fig:RCD} and \ref{fig:SAGA}) and analytically (Theorem \ref{thm:absolute_noise_cst_stepsize} and Proposition \ref{prop:bound_A1_C1}) how the proposed framework improves on existing bounds derived given, as only information on the noise distribution, its expectation, variance, and uncorrelation conditions. 

However, as compared to the all-scenarios approach introduced in \cite{taylor2019stochastic}, which uses, when available, more information on the distribution of the noise than its expectation, variance, and uncorrelation conditions, and hence applies to a more restrictive set of problems, the proposed framework may yield worse guarantees, as it is typically the case for linear convergence rates (Figures \ref{fig:sum_of_functions}, \ref{fig:SAGA}). On the other hand, the all-scenarios approach is restricted to the analysis of finite-support noise settings and, mainly, to a Lyapunov function approach. In addition, it yields better rates only if further information on the noise distribution is available, e.g., in the finite-sum and the block coordinate settings. Hence, on (i) noise models defined only via their expectation, variance, and uncorrelation, e.g., the absolute bounded model (Figure \ref{fig:absolute_noise_cst_stepsize}), or (ii) problems with sublinear convergence rates (Figure \ref{fig:RCDb}), the proposed framework improves on existing results yielded by the all-scenarios approach.

In addition, the framework significantly improves on rates obtained via the worst-scenario approach, i.e., under deterministic perturbation (Figures \ref{fig:absolute_noise_cst_stepsize} and \ref{fig:absolute_noise_cst_stepsize}). This demonstrates the necessity, in situations where the perturbation is in fact stochastic rather than deterministic, to analyze methods under a stochastic framework.

\paragraph{Additive bounded noise model.}  
Consider the maybe simplest setting, i.e., \eqref{eq:SGD} with constant stepsizes $\alpha$, under the additive bounded noise model, i.e., $\E[\|\ve_k\|^2] \leq \sigma^2$. Figure \ref{fig:absolute_noise_cst_stepsize} compares various convergence rates, in both the convex and the strongly convex regimes. Specifically, it displays the worst-case scenario, or equivalently the analysis of \eqref{eq:SGD} under a deterministic bounded perturbation; the PEP-based bounds obtained via the single realization and two realizations approaches; and, in the convex case, the all-scenarios PEP-based bound of \cite[Theorem 5]{taylor2019stochastic} for $\alpha=\frac{1}{L}$, that is
\begin{align}
    \E[f(x_N)-f_\star]\leq \frac{L}{2N}\|x_0-x_\star\|^2+\frac{N+3}{4L}\sigma^2.
\end{align}
In addition, in the strongly convex case, it displays an existing bound from \cite[Theorem 4.1]{gorbunov2020unified}, that is, \begin{align}
    E[\|x_N-x_\star\|^2]\leq \phi\|x_0-x_\star\|^2+\frac{\sigma^2\alpha}{\mu}, \text{ where } \phi=1-\mu \alpha.
\end{align}
This bound is derived based on the weaker assumption $E[\|\ve_k\|^2|\A_k]\leq2L(f(x_k)-f(x_\star))+\sigma^2$, satisfied by in the additive bounded noise setting, see, e.g., \cite[Example 2]{stich2019unified}.

\begin{figure}[ht!]
  \centering
  \begin{subfigure}[t]{0.49\textwidth}
    \centering
    \includegraphics[width=\linewidth]{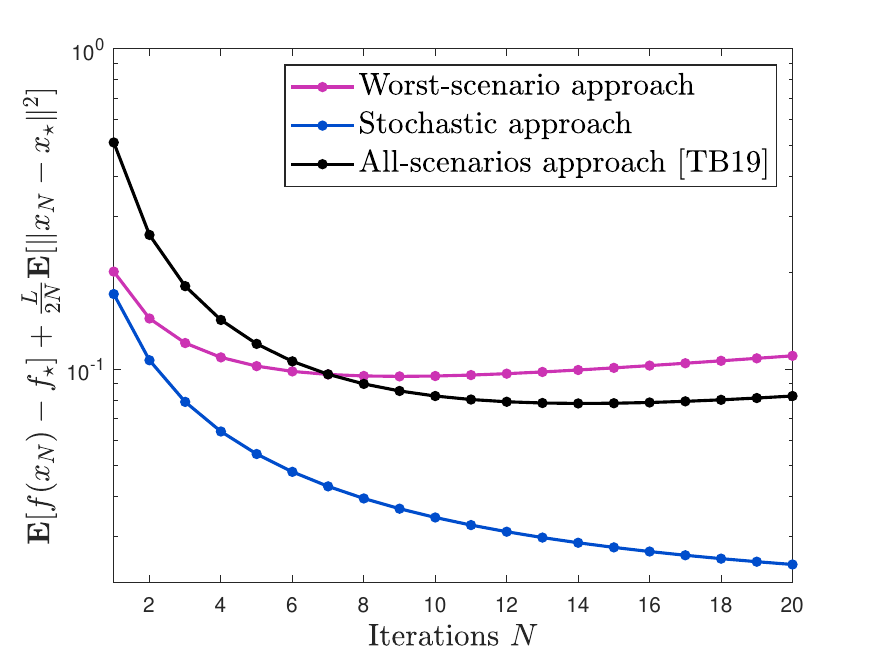}
    \caption{Computer-aided convergence analyses of \eqref{eq:SGD} with constant sepsizes $\alpha=\frac{1}{L}$ on $\F_{0,L}$ under the additive bounded noise model, with performance measure $\E[f(x_N)-f_\star]$, and initial condition $\|x_0-x_\star\|\leq 1$, for $L=1$.}
  \end{subfigure}
  \hfill
  \begin{subfigure}[t]{0.49\textwidth}
    \centering
    \includegraphics[width=\linewidth]{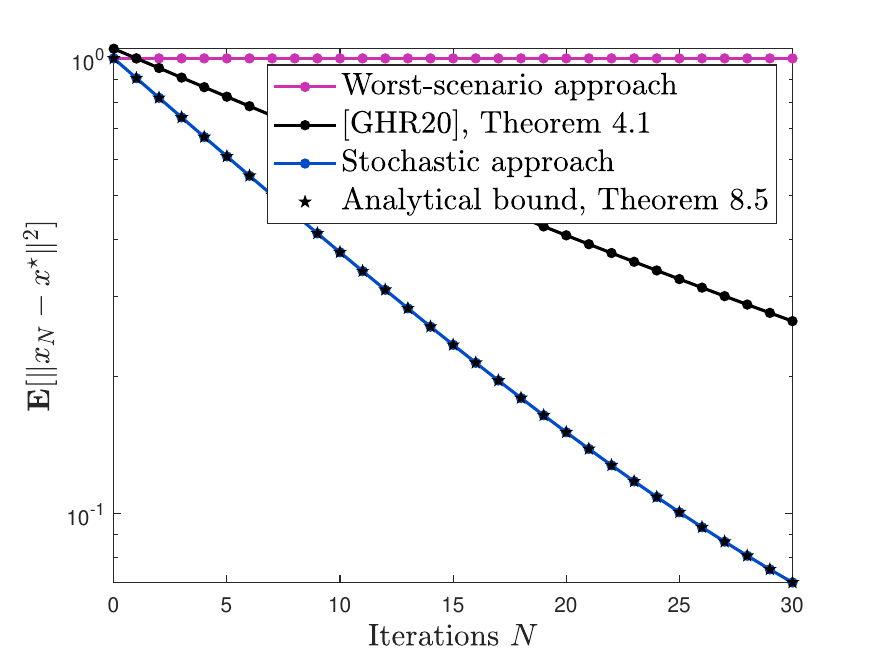}
    \caption{Convergence analyses of \eqref{eq:SGD} with constant sepsizes $\alpha=\frac{1}{2L}$ on $\F_{\mu,L}$, $\mu=0.1$, under the additive bounded noise model, with performance measure $\E[\|x_N-x_\star\|^2]$, and initial condition $\|x_0-x_\star\|\leq 1$, for $L=1$.}
  \end{subfigure}
  \caption{Comparison of computer-aided convergence analyses of \eqref{eq:SGD} with constant stepsizes $\alpha$, on $\F_{0,L}$ (convex $L$-smooth functions, left-hand side figure), and $\F_{\mu,L}$ ($\mu$-strongly convex $L$-smooth functions, $\mu>0$, right-hand side figure), under the additive bounded noise model of variance $\sigma^2=0.01$. Displayed are the worst-scenario PEP-based bound (pink), the PEP-based bound from the framework presented in Section \ref{sec:framework} (blue), the all scenarios PEP-based bound from \cite{taylor2019stochastic}(black), the bound from \cite[Theorem 4.1]{gorbunov2020unified}(black), and the analytical bound from Theorem \ref{thm:absolute_noise_cst_stepsize} ($\star$).}
  \label{fig:absolute_noise_cst_stepsize}
\end{figure}

The bound computed via the proposed framework outperforms \cite[Theorem 1]{moulines2011non} and \cite[Theorem 5]{taylor2019stochastic}. Although \cite[Theorem 5]{taylor2019stochastic} derives from an all-scenarios formulation, it considers as performance measure a slightly different setting, that is, $\E[f(x_N)-f(x_\star)+\frac{L}{2N}\|x_N-x_\star\|^2]$, from which we infer a bound on $\E[f(x_N)-f(x_\star)]$. This explains why it is initially outperformed by the worst-scenario guarantee. 

Finally, in the strongly convex case, relying on numerical insight, we identify an analytical convergence rate that exactly matches the numerical bound from \eqref{eq:relax}, and prove its tightness.
\begin{theorem}\label{thm:absolute_noise_cst_stepsize}
  Let $N\in \N$, $0<\mu<L$, $\sigma^2 \in \R$, $\alpha \in (0,\frac{2}{L+\mu}]$. Consider $N$ iterations of \eqref{eq:SGD} with constant stepsize $\alpha_k=\alpha$, $k=1,\ldots,N$, applied to a function $f\in \F_{\mu,L}$, under additive bounded noise $(\ve_k)_{k\in \lb N-1\rb}$ satisfying $\E[\ve_k|\A_k]=0$, and $\E[\ve_k|\A_k]^2\leq \sigma^2$. Let $(x_k)_{k\in \NN}$ be the iterates of \eqref{eq:SGD}, and $x_\star$ denote a minimizer of $f$. Then, \begin{align}
      E[\|x_{N}-x_\star\|^2]\leq \phi^{2N}\|x_0-x_\star\|^2+\frac{1-\phi^{2N}}{1-\phi^2} \alpha^2\sigma^2,
  \end{align}  
  where $\phi=(1-\mu\alpha)$. In addition, this bound is tight.
\end{theorem}
\begin{proof}
The first part of the proof follows exactly \cite[Section 2.8.1]{taylor2024towards}, which addresses the noise-free case. Let $S=\{(x_i,f_i,g_i)\}_{i\in\NN\cup\star}$ denote the set of iterates (and minimizer) and associated function and gradient values. We start by analyzing a single iteration of \eqref{eq:SGD}, and claim the following.
\begin{align}\label{eq:claim1}
      \|x_{k+1}-x_\star\|^2\leq \phi^{2}\|x_k-x_\star\|^2+\alpha^2 \sigma^2.
  \end{align}  
By Proposition \ref{prop:FmuL_CI} describing smooth strongly convex functions, and since $g_\star=0$, it holds
    \begin{align*}
        &f_\star\geq f_k+\langle g_k,x_\star-x_k\rangle +\frac{L}{2(L-\mu)}\bigg(\frac{\|g_k\|^2}{L}+\mu\|x_k-x_\star\|^2-\frac{2\mu}{L}\langle g_k,x_k-x_\star\rangle\bigg).\\
        &f_k\geq f_\star +\frac{L}{2(L-\mu)}\bigg(\frac{\|g_k\|^2}{L}+\mu\|x_k-x_\star\|^2-\frac{2\mu}{L}\langle g_k,x_k-x_\star\rangle\bigg).
    \end{align*}
Summing these inequalities, multiplied by $2\alpha\phi$, yields
\begin{align*}
&-\frac{2\alpha\phi(L+\mu)}{L-\mu} \langle g_k,x_k-x_\star\rangle +\frac{2\alpha\phi}{L-\mu}\|g_k\|^2+ \frac{2\alpha\phi\mu L}{L-\mu}\|x_k-x_\star\|^2\leq 0\\
   \Leftrightarrow& \|x_k-x_\star\|^2-2\alpha\langle g_k, x_k-x_\star\rangle+\alpha^2\|g_k\|^2\leq \phi^2\|x_k-x_\star\|^2\\&\hspace{6cm}-\frac{\alpha (2-\alpha(L+\mu))}{L-\mu}\|\mu (x_k-x_\star)-g_k\|^2\\
    &\hspace{6cm}\leq \phi^2\|x_k-x_\star\|^2,
\end{align*}
where the second inequality follows from the domain of $\alpha$. Hence,
\begin{align*}
    &\|x_k-x_\star\|^2-2\alpha\langle x_k-x_\star, g_k+\ve_k\rangle+\alpha^2\|g_k+\ve_k\|^2
    \\\leq& \phi^2\|x_k-x_\star\|^2 -2\alpha\langle x_k-x_\star, \ve_k\rangle+\alpha^2\big(2\langle g_k,\ve_k\rangle+\|\ve_k\|^2\big)\\
    \Leftrightarrow&\|x_{k+1}-x_\star\|^2  \leq \phi^2\|x_k-x_\star\|^2 -2\alpha\langle x_k-x_\star, \ve_k\rangle+\alpha^2\big(2\langle g_k,\ve_k\rangle+\|\ve_k\|^2\big).
\end{align*}
Taking the expectation with respect to $\ve_k$ (conditioned on $x_k$) yields
\begin{align*}
    \E_{\ve_k}[\|x_{k+1}-x_\star\|^2] &\leq \phi^2\|x_k-x_\star\|^2 +\alpha^2 \E_{\ve_k}[\|\ve_k\|^2]\leq \phi^2\|x_k-x_\star\|^2 +\alpha^2 \sigma^2,
\end{align*}
by assumption on $\ve_k$. Hence, \eqref{eq:claim1} holds. By recursion, it holds
\begin{align*}
    \E[\|x_{N}-x_\star\|^2]\leq \phi^{2N}\|x_0-x_\star\|^2+\sum_{i=0}^{N-1}\phi^{2i} \alpha^2 \sigma^2.
\end{align*}
Since, for $r\neq1$, $\sum_{i=0}^{N-1} r^i=\frac{1-r^N}{1-r}$, this concludes the derivation of the bound. 

To conclude the proof, observe that the function $f(x)=\frac{\mu}{2}\|x\|^2$ and any distribution satisfying $\E[\ve_k]=\sigma^2, \ k=0,1,\ldots,N$, attain the bound. Indeed, then, $x_\star=0$, and
\begin{align*}
    x_{k+1}=\phi x_k -\alpha \ve_k.
\end{align*}
Hence, 
\begin{align*}
   & (x_{k+1}-x_\star)^2=(\phi (x_k-x_\star) -\alpha \ve_k)^2=\phi^2(x_k-x_\star)^2-2\alpha\phi \ve_k(x_k-x_\star)+\alpha^2\ve_k^2 \\
    \Leftrightarrow& \E[(x_{k+1}-x_\star)^2|\A_k]=\phi^2(x_k-x_\star)^2+\alpha^2\sigma^2.
\end{align*}
\end{proof}
\paragraph{Finite-sum setting: on the importance of the parameters in \eqref{eq:assum_noise}.}
We now turn to the analysis of \eqref{eq:SGD} with constant stepsizes, under noise satisfying a variance condition of the form\begin{align}
    \E[\|\ve_k\|^2|\A_k] \leq A_1(f(x_k) - f(x_\star)) + D_1.\label{eq:assum_finite_sum}
\end{align}
This includes, in particular, the finite-sum setting where \(f = \sum_i f_i\) is a sum of convex \(L_i\)-smooth functions with \(L = \max_i L_i\), and the variance at the optimum is bounded, i.e., \(\sum_i \|\nabla f_i(x_\star)\|^2 := \sigma_\star^2 < +\infty\). Indeed, as shown in, e.g.,~\cite[Lemma 4.20]{garrigos2023handbook}, this setup satisfies~\eqref{eq:assum_finite_sum} with \(A_1 = 4L\) and \(D_1 = 2\sigma_\star^2\). Moreover, when \(\sigma_\star^2 = 0\), e.g., in overparametrized models \cite{ma2018power}, tighter constants can be used: \(A_1 = 2L\), \(D_1 = 0\); see~\cite[Lemma 4.19]{garrigos2023handbook}.  

In \cite[Theorem 4.1]{gorbunov2020unified}, a convergence rate is derived on the basis of \eqref{eq:assum_finite_sum}, that is, using only the expectation, variance, and uncorrelation properties of the noise. This bound is given by
\begin{align}
      E[\|x_{N}-x_\star\|^2]\leq \phi^{N}\|x_0-x_\star\|^2+\frac{\alpha^2D_1}{\alpha\mu}, \text{ where } \phi=(1-\mu\alpha).
  \end{align} 
On the other hand, the finite-sum setting was analyzed using the exact all-scenarios approach in~\cite{taylor2019stochastic} and~\cite{cortild2025new}, using all information available on the noise distribution (hence, for a more restrictive setting than \eqref{eq:relax}). In particular, \cite[Theorem 3.3]{cortild2025new} shows 
\begin{align}
      E[\|x_{N}-x_\star\|^2]\leq \phi^{N}\|x_0-x_\star\|^2+\frac{1-\phi^N}{1-\phi}\frac{\alpha^2D_1}{2}\bigg(1+\frac{\alpha (L- \mu)}{2-(L+\mu)\alpha}\bigg), 
  \end{align} 
where $\phi=(1-\mu\alpha)^2$.
Figure~\ref{fig:sum_of_functions} compares the proposed framework with these existing rates from ~\cite[Theorem 4.1]{gorbunov2020unified} and~\cite[Theorem 3.3]{cortild2025new}. The right-hand side of Figure~\ref{fig:sum_of_functions} illustrates the impact of the parameter choice in the variance model. For instance, in the case \(\sigma_\star^2 = 0\), using \(A_1 = 4L\) (as in the general case) rather than the sharper \(A_1 = 2L\) yields looser performance guarantees.

\begin{figure}[ht!]
  \centering
  \begin{subfigure}[t]{0.49\textwidth}
    \centering
    \includegraphics[width=\linewidth]{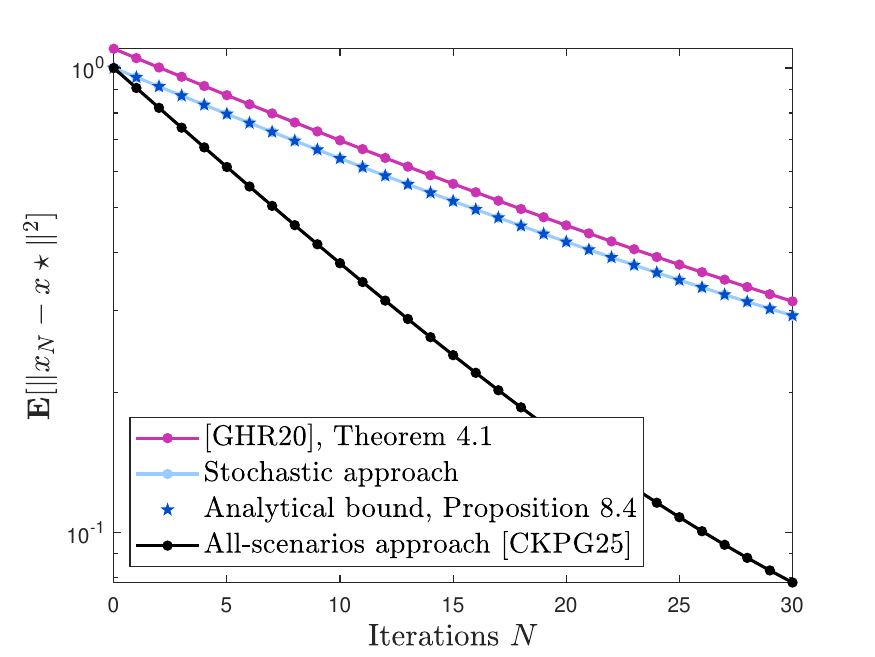}
  \end{subfigure}
  \hfill
  \begin{subfigure}[t]{0.49\textwidth}
    \centering
    \includegraphics[width=\linewidth]{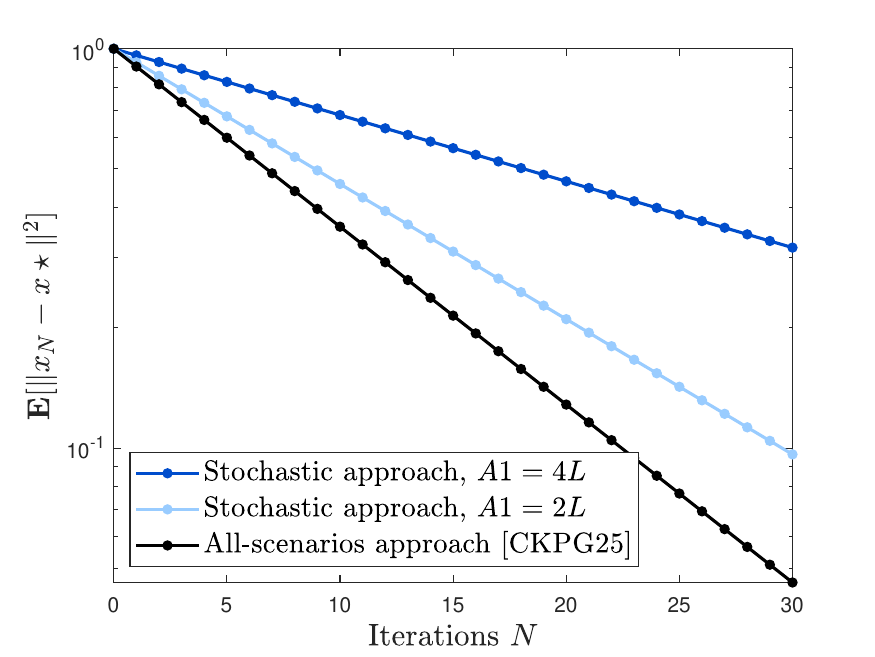}
  \end{subfigure}
  \caption{Comparison of convergence analyses of \eqref{eq:SGD} with constant stepsizes $\alpha=\frac{1}{2L}$, on $\F_{\mu,L}$, under noise satisfying $\E[\ve_k|\A_k]=0$, $\E[\|\ve_k\|^2|\A_k]\leq A_1(f(x_k)-f_\star)+D_1$. The performance criterion is $\E[\|x_N - x_\star\|^2]$, with initial condition $\|x_0 - x_\star\| \leq 1$, and parameters $\mu=0.1$, $L=1$, and $D_1=0.01$ (left-hand side figure) or $D_1=0$ (right-hand side figure). Displayed are the bound from \cite[Theorem 4.1]{gorbunov2020unified}, based solely on the expectation, variance, and uncorrelation of the noise (pink), the PEP-based bounds from \eqref{eq:relax} (blue), the analytical bound from Proposition \ref{prop:bound_A1_C1} ($\star$), and the exact all-scenarios PEP-based bound from \cite[Theorem 3.3]{cortild2025new}, considering the more restrictive setting of finite-sum optimization (black).}
  \label{fig:sum_of_functions}
\end{figure}
The bound \cite[Theorem 4.1]{gorbunov2020unified} is obtained by (i) showing satisfaction, for all $k\in[N-1]$, of
 \begin{align}\label{eq:gorbu}
      E[\|x_{k+1}-x_\star\|^2]\leq \phi\|x_k-x_\star\|^2+\alpha^2D_1, \text{ where } \phi=(1-\mu\alpha)\ \text{\cite[Lemma C1]{gorbunov2020unified}},
  \end{align}  
and (ii) recursively applying \eqref{eq:gorbu}. In the computations, $\sum_{i=0}^{N-1} \phi^i$ is relaxed into $\frac{1}{1-\phi}$. We strengthen \cite[Theorem 4.1]{gorbunov2020unified} by removing this relaxation.
\begin{proposition}\label{prop:bound_A1_C1}
      Let $N\in \N$, $0<\mu<L$, $A_1,D_1\in \R$, $\alpha \in (0,\frac{2}{A_1}]$. Given a function $f\in \F_{\mu,L}$, denote by $x_\star$ one of its minimizers. Consider $N$ iterations of \eqref{eq:SGD} with constant stepsize $\alpha_k=\alpha$, $k=1,\ldots,N$, on $f$, with additive noise $(\ve_k)_{k\in \lb N-1\rb}$ satisfying $\E[\ve_k|\A_k]=0$, and $\E[\|\ve_k\|^2|\A_k]\leq A_1(f(x_k)-f(x_\star))+D_1$, where $(x_k)_{k\in \NN}$ denote the iterates of \eqref{eq:SGD}. Then, \begin{align}
      E[\|x_{N}-x_\star\|]^2\leq \phi^{N}\|x_0-x_\star\|^2+\frac{1-\phi^{N}}{1-\phi} \alpha^2D_1,
  \end{align}  
  where $\phi=(1-\mu\alpha)$.
\end{proposition}
\begin{proof}
    It suffices to perform a recursion on \eqref{eq:gorbu}, and rely on $\sum_{i=0}^{N-1} \phi^i=\frac{1-\phi^N}{1-\phi}$.
\end{proof}

The analytical guarantee in Proposition~\ref{prop:bound_A1_C1} coincides with the numerical bound from~\eqref{eq:relax}, for $\alpha=\frac{2}{A_1}$. Additionally, the bound from~\cite[Theorem 4.1]{gorbunov2020unified} matches exactly when \(D_1 = 0\). As expected, the exact bound under the finite-sum setting is tighter than those based solely on expectation, variance, and uncorrelation of the noise. In addition, multiple parameter choices for the noise variance are possible in the finite-sum setting, resulting in bounds whose quality depends on these parameters. However, the bounds from \cite[Theorem 4.1]{gorbunov2020unified}, \eqref{eq:relax}, and Proposition~\ref{prop:bound_A1_C1}, cover more stochastic settings than finite-sum problems. 
\paragraph{Strong growth condition and RCD}
We now consider noise satisfying the strong growth condition, i.e., for some $d\in\N$, $$\E[\|\ve_k\|^2|\A_k] \leq (d-1)\|g_k\|^2,$$ introduced in \cite{schmidt2013fast}. This condition is satisfied by several settings, including some finite-sum problems, where it implies $\E[\|\ve_k\|^2|\A_k] \leq 2Ld(f(x_k) - f(x_\star))$, see \cite[Proposition 1]{vaswani2019fast}, and randomized coordinate descent (RCD), where the gradient estimator is $d\e_i^\top \nabla f(x_k) \e_i$, with $d$ the dimension and $\e_i$ a randomly selected basis vector \cite[Section 1.4]{wright2015coordinate}. This estimator is chosen as to be zero-mean, and the classical estimator $\e_i^\top \nabla f(x_k) \e_i$ can be analyzed via a modified stepsize, $\tilde{\alpha}=d\alpha$. Throughout, we set $\alpha=\frac{1}{d}$.

Figure~\ref{fig:RCDa} compares, in the strongly convex regime, the bound resulting from the proposed framework, the bound from \cite[Theorem 4.1]{gorbunov2020unified} (based on the variance model $\E[\|\ve_k\|^2|x_k] \leq 2Ld(f(x_k) - f(x_\star))$), and the exact all-scenarios approach from \cite[Theorem 17]{taylor2019stochastic}, specifically adapted to RCD. This bound is given by 
\begin{align}
    \E[\|x_N-x_\star\|^2]\leq\bigg(\frac{(1-\frac{\mu}{L})^2+d-1}{d}\bigg)^N\|x_0-x_\star\|^2.
\end{align}
The bound from \cite{abbaszadehpeivasti2022convergence}, solving a slightly relaxed version of \eqref{eq:relax}, exactly matches \eqref{eq:relax}. 

Figure \ref{fig:RCDb} compares, in the convex regime, the bound resulting from the proposed framework, the bound resulting from the approach proposed in \cite{abbaszadehpeivasti2022convergence}, and the all-scenarios approach from \cite[Theorem 17]{taylor2019stochastic}, i.e., letting $\tilde{\alpha}=\frac{1}{L}$,
\begin{align}
    \E[f(x_N)-f_\star]\leq \frac{d\tilde{\alpha}}{2N}\|x_0-x_\star\|^2.
\end{align}
\begin{figure}[ht!]
  \centering
  \begin{subfigure}[t]{0.49\textwidth}
    \centering
    \includegraphics[width=\linewidth]{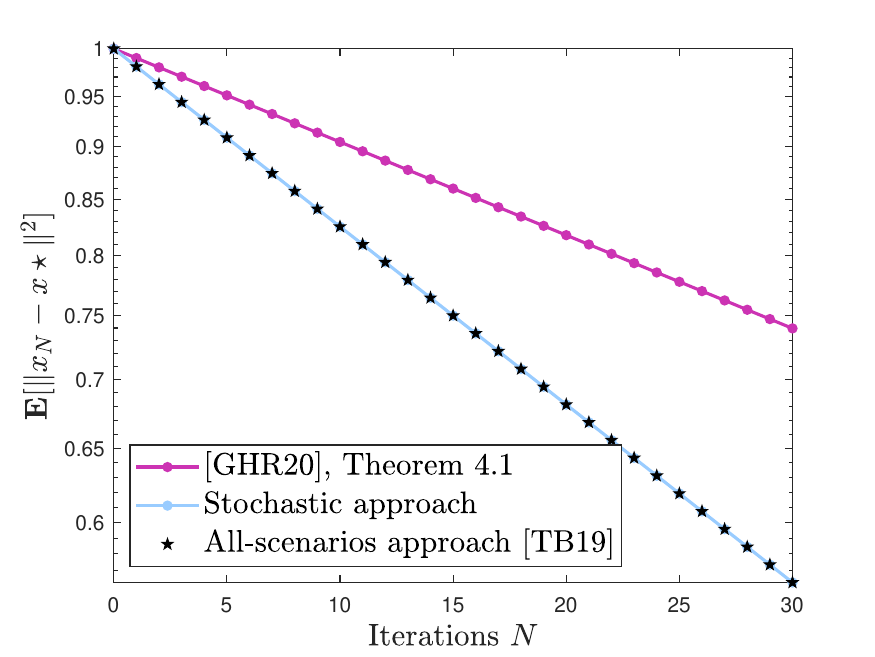}
    \caption{Strongly convex regime}
    \label{fig:RCDa}
  \end{subfigure}
  \hfill
  \begin{subfigure}[t]{0.49\textwidth}
    \centering
    \includegraphics[width=\linewidth]{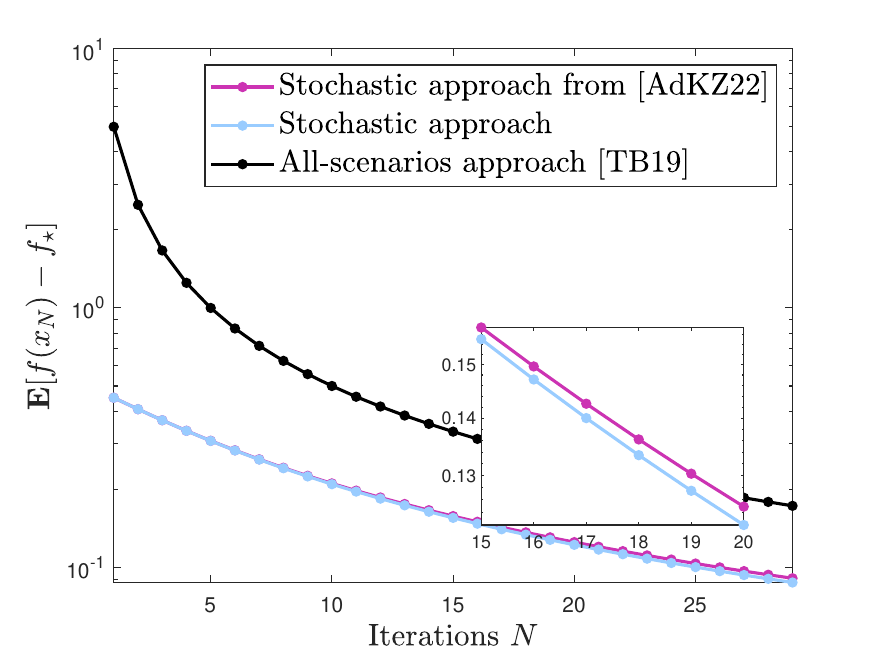}
    \caption{Convex regime}
    \label{fig:RCDb}
  \end{subfigure}
  \caption{Comparison of convergence analysis of \eqref{eq:SGD} with constant stepsize $\alpha_k = \frac{1}{dL}$, on $\F_{\mu,L}$ ($\mu>0$) and $\F_{0,L}$, with noise satisfying $\E[\ve_k|\A_k] = 0$, and $\E[\|\ve_k\|^2|\A_k] \leq (d - 1)\|g_k\|^2$. The performance measure is $\E[\|x_N - x_\star\|^2]$ with initial condition $\|x_0 - x_\star\| \leq 1$, for $\mu = 0.1$, $L = 1$, and $d = 10$. The left-hand side figure, considering the strongly convex regime, displays the bound from \cite[Theorem 4.1]{gorbunov2020unified}, using $\E[\|\ve_k\|^2|x_k] \leq 2Ld(f(x_k) - f(x_\star))$ (pink), the PEP-based bound from \eqref{eq:relax} (blue), and the exact all-scenarios bound designed for RCD from \cite[Theorem 17]{taylor2019stochastic} (black). The right-hand side figure, considering the convex regime, displays the PEP-based bound from \eqref{eq:relax} (blue), the PEP-based bound from \cite{abbaszadehpeivasti2022convergence} (pink), and the all-scenarios bound designed for RCD from \cite[Theorem 16]{taylor2019stochastic} (black).}
    \label{fig:RCD}
\end{figure}

Surprisingly, the bounds resulting from the all-scenario approach, i.e., \cite[Theorems 16, 17]{taylor2019stochastic}, and that rely on more information than the framework of Section \ref{sec:framework}, (i) exactly match the bounds resulting from \eqref{eq:relax}, in the strongly convex setting, and (ii) are outperformed by \eqref{eq:relax} in the convex setting. This lack of tightness is explained by the structure of the Lyapunov function, chosen by hand in the all-scenario approach. In addition, the bound from \eqref{eq:relax} improves on \cite[Theorem 4.1]{gorbunov2020unified}, and, in the convex setting (sublinear convergence rate), on the one of \cite{abbaszadehpeivasti2022convergence}.
\paragraph{Variance-reduction method: SAGA}

Finally, we consider the variance-reduction method SAGA \cite{defazio2014saga}, designed for finite-sum optimization problems. SAGA update at iteration \(k\) is
\[
x_{k+1} = x_k - \alpha \left( \nabla f_{i_k}(x_k) - y_{i_k}^k + \frac{1}{n} \sum_{j=1}^n y_j^k \right),
\]
where \(i_k\) is sampled uniformly from \(\{1, \ldots, n\}\), and \(y_j^k\) stores the most recent gradient \(\nabla f_j\) evaluated at some previous iterate. SAGA can be embedded as \eqref{eq:SGD}, with constant stepsizes $\alpha$, and $\nabla f_{i_k}(x_k) - y_{i_k}^k + \frac{1}{n} \sum_{j=1}^n y_j^k $ an unbiaised estimator of $\nabla f(x_k)$, satisfying \(\mathbb{E}[\ve_k|\A_k] = 0\), and the variance condition \eqref{eq:assum_noise} with parameters
\[
A_1 = 4L,\quad  B_1=-1, \quad E_1 = 2, \quad \rho = \frac{1}{n}, \quad A_2 = \frac{2L}{n},
\]
where \(L = \max_i L_i\), \(n\) is the number of component functions, and $\sigma_k^2=\frac{1}{n}\sum_{i=1}^n\|y_i^k-\nabla f_i(x_\star)\|^2$ \cite[Lemma A.6]{gorbunov2020unified}.

Figure~\ref{fig:SAGA} compares the convergence bounds from \cite[Theorem 4.1]{gorbunov2020unified} (based only on noise expectation and variance), the PEP-based bound from \eqref{eq:relax}, and the exact all-scenarios bound after one iteration, from \cite{pepit2022}. The exact all-scenarios bound outperforms \eqref{eq:relax} but suffers from scalability issues. In addition, it depends on \(n\), making analytical characterization difficult. The bound from \cite[Theorem 4.1]{gorbunov2020unified} coincides with \eqref{eq:relax} after one iteration, but loses its tightness property from $2$ iterations.
\begin{figure}[ht!]
    \centering
    \includegraphics[width=0.55\linewidth]{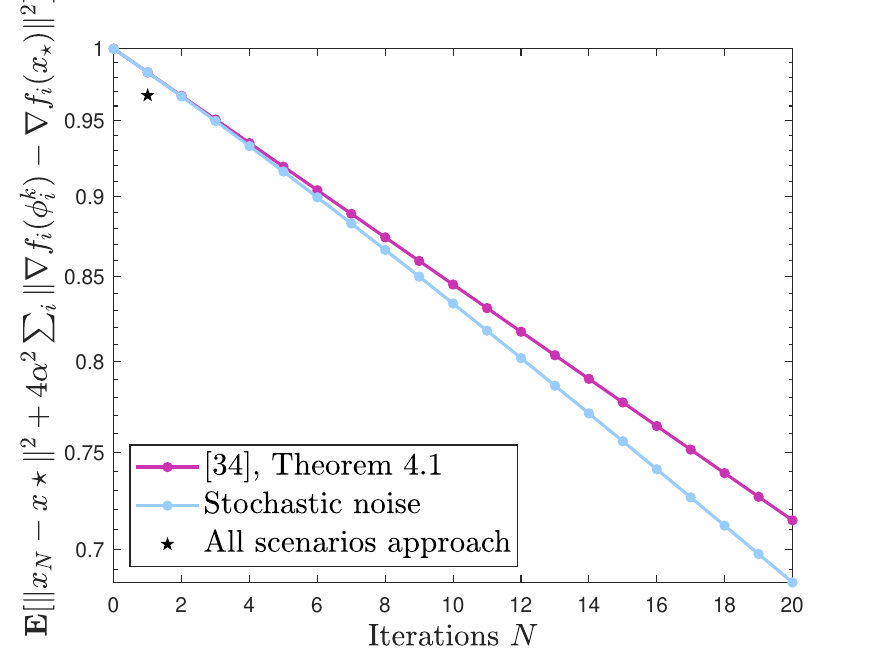}
\caption{Comparison of convergence analyses of \eqref{eq:SGD} with constant stepsizes $\alpha=\frac{1}{6L}$, on $\F_{\mu,L}$, with noise satisfying $\E[\ve_k|\A_k] = 0$, and \eqref{eq:assum_noise} with parameters $A_1 = 4L$, $E_1 = 2$, $\rho = \frac{1}{n}$, and $A_2 = \frac{2L}{n}$. These assumptions hold for SAGA \cite{defazio2014saga}. We consider the performance measure $\E[\|x_N - x_\star\|^2 + 4n \alpha^2 \sigma_N^2]$ with initial condition $\|x_0 - x_\star\|^2 + 4n \alpha^2 \sigma_0^2 \leq 1$, for $\mu = 0.1$, $L = 1$, and $n = 10$. Displayed are the bound from \cite[Theorem 4.1]{gorbunov2020unified} (pink), the PEP-based bound from \eqref{eq:relax} (blue), and the exact all-scenarios PEP-based bound ($\star$).}
    \label{fig:SAGA}
\end{figure}
\section{Conclusion}
We have proposed a framework for automated analysis of stochastic first-order methods, under all noise models satisfying Assumption~\ref{assum:prob_def}, and independently of the noise's distribution. The framework positions itself between PEP analyses tailored for deterministic perturbation \cite{taylor2017exact,de2020worst,vernimmen2025empirical,vernimmen2025worst} (or equivalently, analyzing the worst realization of a stochastic perturbation) and those, regarding finite support settings, explicitly enumerating all scenarios \cite{cortild2025new,taylor2019stochastic,kamri2023worst}. Indeed, by embedding expectations as variables within the PEP framework, as in \cite{abbaszadehpeivasti2022convergence}, our method improves upon deterministic analyses while maintaining similar complexity. The framework also compares to the all-scenarios approach. On structural noise models, e.g., finite-sum setting or block-coordinate settings, it may yield looser rates than those obtained via the all-scenario approach, relying on this structural information. This is especially the case for problems with linear convergence rates. However, on problems with sublinear rates, it occasionally improves on the all-scenarios approach, since the latter is (mainly) restricted to Lyapunov-based analyses, while our framework remains tractable on a large number of iterations, and yields stronger convergence rates than those obtained via a Lyapunov approach. In addition, by contrast to the all-scenarios approach, the proposed framework enables a
unified framework applicable across diverse stochastic settings, and is not restricted to finite supports.

\paragraph{Extensions}
While this work focused on \eqref{eq:SGD} over the class $\F_{\mu,L}$, the proposed framework readily extends to a broader range of methods and function classes. As detailed in Section~\ref{sec:framework}, any problem satisfying Assumption~\ref{assum:prob_def}, including accelerated methods or those applied to non-convex or non-smooth problems, can be analyzed within this framework. Aside from analyzing existing methods, the framework can also be relied upon to design new efficient methods, as illustrated, e.g., in \cite{drori2020efficient}.

Future work also includes obtaining tightness guarantees for the proposed framework.

Finally, our framework supports a reverse perspective: identifying the noise assumptions under which desirable convergence rates are achievable, and designing algorithms that satisfy them. 
\section*{Acknowledgments}
The authors would like to thank Balasz Gerencsér for fruitful discussions. 
\section*{Conflict of interest}

The authors declare that they have no conflict of interest. Views and opinions expressed are those of the authors only.
\bibliographystyle{spmpsci}
\bibliography{bibli}
\appendix
\section{Full SDP formulation of an instance of \eqref{eq:relax}}\label{app:fullLMI}
This section provides the full SDP formulation for solving a slightly generalized version of \eqref{eq:PEP_SDP_ex}. That is, we provide an example of SDP formulation for solving \eqref{eq:relax}, on the analysis of $N$ iterations of \eqref{eq:SGD}, on $\F_{0,L}$, under the additive bounded noise model. 

Consider first the general case of analyzing $N$ iterations of a method $\M$ on a class $\F$, under noise $\Omega_N$, and given an initial condition $\Ci$ satisfying Assumption \ref{assum:prob_def}. In Section \ref{sec:framework}, we defined the Gram representation of sets $S = \{(x_i, f_i, g_i)\}_{i \in \NN\cup\star}\in(\R^d\times\R\times\R^d)^{N+2}$, and $V=\{(\ve_i,\sigma_i^2)\}_{i \in[N-1]}\in(\R^d\times\R^d)^N$ as
\begin{align}        
    F &:=[f_\star, f_0, \ldots,\ f_N,\sigma_0^2,\ldots,\sigma_{N-1}^2]^\top,\\
    G&:=PP^\top,\text{ where } P:=[x_\star,x_0,\ldots,x_N,g_\star,g_0,\ldots,g_N,\ve_0,\ldots,\ve_{N-1}]^\top,\nonumber
\end{align}
see \eqref{eq:equivFGSV}. For computational purposes, we introduce a slightly different Gram representation of $S\cup V$, that is,
\begin{align}        
    F &:=[f_0, \ldots,\ f_N,\sigma_0^2,\ldots,\sigma_{N-1}^2]^\top,\\
    G&:=PP^\top,\text{ where } P:=[x_0,g_0,\ldots,g_N,\ve_0,\ldots,\ve_{N-1}]^\top,\nonumber
\end{align}
This representation allows (i) reducing the size of the variables in all SDP formulations of \eqref{eq:PEP_untractable_gen}, and (ii) encoding the optimal condition $\|g_\star\|^2=0$, and the method's dynamics directly via \emph{selection vectors}, whose purpose are to access all scalar products in $G$, including those involving $x_k$, $k\leq 1$. Specifically, we consider w.l.o.g. $x_\star$ to be the origin, and $f_\star=0$. To access elements of $G$ and $F$, given $\e_i$ the $i^{\th}$-unit vector, we then define selection vectors
\begin{itemize}
    \item $\g_\star = \x_\star = \mathbf{0} \in \mathbb{R}^{2N+2}$, $\f_\star=\mathbf{0} \in \mathbb{R}^{2N+1}$,
    \item For $k \in \NN$, $\g_k = \e_{k+2} \in \mathbb{R}^{2N+2}$ and $\f_k = \e_{k+1} \in \mathbb{R}^{N+1}$,
    \item For $k \in \lb N-1\rb$, $\vve_k = \e_{N+3+k} \in \mathbb{R}^{2N+2}$ and $\s_k = \mathbf{e}_{N+2+k} \in \mathbb{R}^{2N+1}$,
    \item $\x_0 = \e_1 \in \mathbb{R}^{2N+2}$, and $\x_{k+1} = \x_0 + \sum_{i=0}^{k+1} \alpha_{k,i} (\g_i + \vve_i), \quad k \in \lb N-1\rb.$
\end{itemize}
These vectors allow accessing all $f_i$'s, $\sigma_j$  ($i\in\NN\cup\star$, $j\in[N-1]$) in $F$, and scalar products of $x_i,g_i,\ve_j$'s ($i\in\NN\cup\star$, $j\in[N-1]$) in $G$. For instance, 
\begin{align*}
    &\x_i^\top G\x_j=\langle x_i,x_j \rangle,&& \g_i^\top G\g_j=\langle g_i,g_j \rangle,     & \vve_i^\top G\vve_j =\langle \ve_i,\ve_j \rangle ,\\
    &\x_i^\top G\g_j=\langle x_i,g_j \rangle,
     && \g_i^\top G\vve_j=\langle g_i,\ve_j \rangle,\\
    &\x_i^\top G\vve_j=\langle x_i,\ve_j \rangle,&&
\end{align*}
and 
\[
    F^\top (\f_i - \f_\star) = f_i - f_\star.
\]
Using these selection vectors allows formulating \eqref{eq:PEP_SDP_ex} (generalized to $N$ iterations) as follows.
    \begin{align*}
    & &&\max_{\substack{(F,G) \in \R^{N+1}\times S_+^{2N+2}}}
    F^\top(\f_N-\f_\star) \text{ s.t. }\\&\text{$\F$-interpolability:}&&
  F^\top (\f_i-\f_j)\geq \g_j^\top G(\x_i-\x_j)+\frac{1}{2L}(\g_j-\g_i)^\top G(\g_j-\g_i), \ i,j \in \NN\cup\star,\\
  &\text{$\Ci$:}&& (\x_0-\x_\star)^\top G(\x_0-\x_\star)\leq 1,\\
  &\Cind:&&\vve_k^\top G\x_j=\vve_k^\top G\g_j=\ve_k^\top G\vve_l=0, \ k,l\in[N-1],\ l\neq k,\ j\leq k, \\&\Cn:&&\vve_k^\top G \vve_k = \sigma^2, \ k\in[N-1],\end{align*}
  where satisfaction of $\M$ arises from the definition of the selection vectors, and the interpolation conditions follow from \eqref{eq:FmuL_CI}. 
\end{document}